\documentclass{article}
\usepackage{fullpage}

\usepackage[a4paper, total={5.5in, 8.5in}]{geometry}

\usepackage{amsmath,amsthm,amssymb}
\usepackage{graphicx}
\usepackage{scrextend}
\usepackage[usenames,dvipsnames]{xcolor}
\usepackage{lineno}
\graphicspath{ {images/} }
\usepackage{hyperref} 
\usepackage{thm-restate}
\usepackage{enumitem}

\newcommand{\CN}{\mathrm{CN}}
\newcommand{\sd}{\,\triangle\,}

\newcommand{\aut}{\operatorname{Aut}}

\newtheorem{theorem}{Theorem}

\newtheorem{lemma}[theorem]{Lemma}

\theoremstyle{definition}
\newtheorem{definition}[theorem]{Definition}

\newtheorem{remark}[theorem]{Remark}

\begin{document}

\title{Automorphisms of Token Graphs That Send $4$-Cycles Generated by Two Edges to Cycles Generated by a $4$-Cycle and Two Tokens}
\author{Ruy Fabila-Monroy \thanks{Departamento de Matem\'aticas, CINVESTAV.} 
\footnote{\tt{ruyfabila@math.cinvestav.edu.mx}}    
\and Sergio Gerardo G\'omez-Galicia\footnotemark[1] \footnote{\texttt{sgomez@math.cinvestav.mx}} 
\and Ana Laura Trujillo-Negrete\thanks{Instituto de Matem\'aticas, UNAM, Mexico.
\texttt{atrujillo@im.unam.mx}}}
\maketitle
\begin{abstract}
Let $G$ be a connected graph. The $k$-token graph of $G$ is the graph $F_k(G)$ whose vertex set consists of all subsets of $k$
vertices of $G$, where two of them are adjacent whenever their symmetric difference is an edge of $G$.
Every automorphism of $G$ induces one of $F_k(G)$, as does
complementation when $k=|G|/2$; automorphisms of this form are called
\emph{induced}. Fabila-Monroy et al.\ (Graphs and Combinatorics 42, 2026) show
that token graphs can have many non-induced automorphisms, arising from
\emph{twin cuts}. These are cut sets $\{x,y\}$
whose two vertices have the same neighbours (other than themselves) in $G$.
These non-induced automorphisms send configurations with a prescribed
number of tokens on each component of $G\setminus \{x,y\}$, totalling $k-1$, and
exactly one token on one vertex of $\{x,y\}$, to the configuration obtained by moving (\emph{flipping}) the token
at $\{x,y\}$ to the other vertex of $\{x,y\}$.
These automorphisms send an induced $4$-cycle generated by moving two tokens on
two disjoint edges to one generated by moving two tokens on a $4$-cycle; thus introducing
what we call a \emph{twist}.
We prove a partial converse: if an isomorphism $\varphi\colon F_k(G)\to F_{k'}(G')$ sends
some $4$-cycle generated by moving two tokens on two disjoint edges to one generated by moving two
tokens on a $4$-cycle, then $G$ and $G'$ have twin cuts $\{x,y\}$ and
$\{x',y'\}$, respectively. We also show that any twist can be undone by
composing $\varphi$ with twin-cut flips.
\end{abstract}
\tableofcontents

\section{Introduction}

Let $G$ be a connected graph on $n$ vertices and $1 \le k \le |G|-1$ an integer.
The $k$-token graph of $G$ is the graph $F_k(G)$ whose vertex set consists of all subsets of $k$
vertices of $G$, where two of them are adjacent whenever their symmetric difference is an edge of $G$.
Token graphs have been defined many
times~\cite{Johns,double_vertex,ktuple,rudolph,godsil,Token}, some of these have been independent.
In this paper we follow the notation and nomenclature of~\cite{Token}.

Let $f$ be an automorphism of $G$. The function that maps
every $\{a_1,\dots,a_k\} \in V(F_k(G))$ to
\[\iota(f)({a_1,\dots,a_k}):=\{f(a_1),\dots,f(a_k)\}\]
is an automorphism of $F_k(G)$. In fact, it was shown in~\cite{sofiamanuel},
that $\iota(\cdot)$ is an injective group homomorphism. Thus,  $\aut(G)$ is always a
subgroup of $\aut(F_k(G))$; where $\aut(G)$ and $\aut(F_k(G))$ are
the automorphism groups of $G$ and $F_k(G)$, respectively.

When $k=n/2$, the function $\mathfrak{c}$ that sends every $A \in F_k(G)$ to its complement
$V(G)\setminus A$ is also an automorphism of $F_k(G)$. It can be verified~\cite{C4Diamond}
in this case, that  $\mathfrak{c}$ is not of the form $\iota(f)$ for
any automorphism $f \in \aut(G)$. Moreover,
$\mathfrak{c}$ commutes with every member of $\iota(\aut(G))$.
Therefore, when $k=n/2$, the subgroup $\langle\iota(\aut(G)),\mathfrak{c}\rangle$
generated by the induced automorphisms and the complement is isomorphic to
$\iota(\aut(G))\times\mathbb{Z}_2$.
We say that an automorphism in $\iota(\aut(G))$, or in $\langle\iota(\aut(G)),\mathfrak{c}\rangle$ when $k=n/2$, is \emph{induced}.

However, $\aut(F_k(G))$ can have many non induced automorphisms. In~\cite{2t_cube}, it has been shown that if $G$ is the Cartesian product of $r$ prime graphs, then
\[
    \mathbb{Z}_{2}^{r-1}\rtimes\aut(G)\leq \aut(F_{2}(G)).
\]

Recently, in~\cite{cut_aut} the authors construct connected graphs whose token graphs have many
non-induced automorphisms. In this paper we prove a partial converse: an isomorphism between token
graphs satisfying a certain local condition forces the twin-cut structure of~\cite{cut_aut}, which we
now describe.

\subsection{Twin cuts and twisted $4$-cycles}

A pair of vertices $x,y$ of $G$ are called \emph{twins} if $N(x)\setminus \{y\}=N(y)\setminus \{x\}$; that is if they have the same neighbours (other than themselves); if in addition $\{x,y\}$ is a cut set
of $G$ we say that $\{x,y\}$ is a \emph{twin cut}.

Suppose that $\{x,y\}$ is a twin cut, and let $C_1,\dots,C_r$ be the components
of $G\setminus\{x,y\}$. Let $\sigma:=(k_1,\dots,k_r)$ be a tuple of non-negative
integers with $\sum_{i=1}^{r}k_i=k-1$ and $0\le k_i\le|C_i|$ for every
$1\le i\le r$.
For a configuration $A\in V(F_k(G))$ with $|A\cap\{x,y\}|=1$, let
\[
  A^{xy}:=(A\setminus\{x,y\})\cup(\{x,y\}\setminus A)
\]
be the configuration obtained by moving the token on $\{x,y\}$ to the other
vertex of $\{x,y\}$. Let $\varphi_\sigma\colon V(F_k(G))\to V(F_k(G))$ be the
function that sends every $A\in V(F_k(G))$ to
\[
  \varphi_\sigma(A):=
  \begin{cases}
    A^{xy} & \text{if } |A\cap\{x,y\}|=1 \text{ and } |A\cap C_i|=k_i
      \text{ for every } 1\le i\le r,\\[2pt]
    A & \text{otherwise.}
  \end{cases}
\]
It can be shown that $\varphi_\sigma$ is a non-induced automorphism of $F_k(G)$.

Let $C$ be an induced $4$-cycle of $F_{k}(G)$.
It has been shown in~\cite[Proposition~4.1]{C4Diamond} that $C$ is generated
in one of the four ways depicted in Figure~\ref{fig:ways}.
We say that induced $4$-cycles of $F_{k}(G)$ generated as in $(1)$, $(2)$ or $(3)$ of Figure~\ref{fig:ways}
are \emph{generated by moving:} one, two, or three tokens on a $4$-cycle, respectively.
We say that induced $4$-cycles of $F_{k}(G)$ generated as in $(4)$ of Figure~\ref{fig:ways} are
\emph{generated by disjoint edges}.

\begin{figure}[t]
	\centering
	\includegraphics[width=0.9\textwidth]{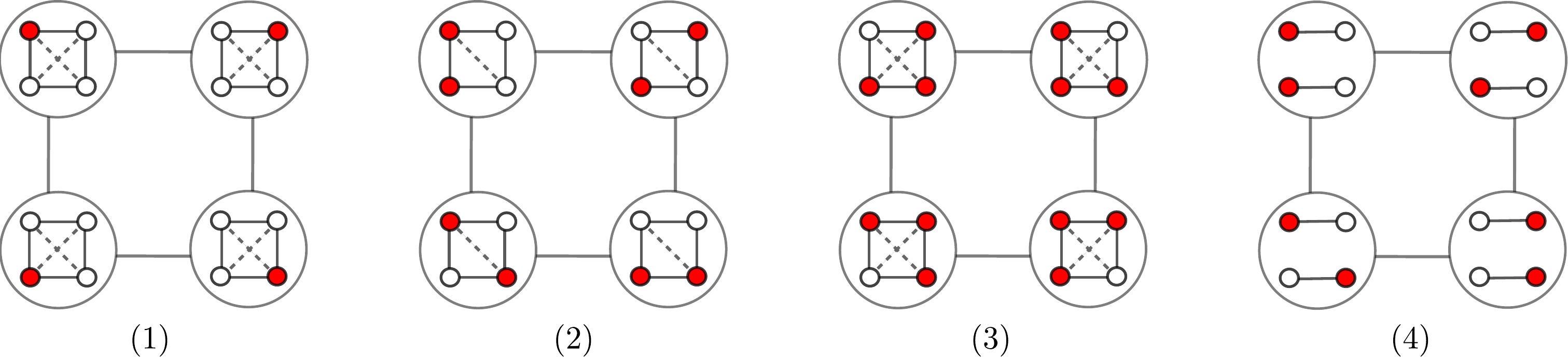}
	\caption{The four ways in which an induced $4$-cycle of $F_k(G)$ can be generated; tokens that are not shown are
assumed to remain fixed; and dashed lines are forbidden in $G$.}\label{fig:ways}
\end{figure}

The flip $\varphi_\sigma$ sends an induced $4$-cycle generated by disjoint edges to one generated by
moving two tokens on a $4$-cycle. In this paper we show a converse: whenever an isomorphism between
token graphs exhibits this behaviour, the underlying graphs carry a twin cut, and the behaviour can be
undone by a twin-cut flip. We work in the general setting of an isomorphism between $F_k(G)$ and the
token graph $F_{k'}(G')$ of another graph $G'$.

Let $\varphi$ be an isomorphism from $F_k(G)$ to $F_{k'}(G')$ that sends
an induced $4$-cycle, $(A_1,A_2,A_3,A_4)$,
generated by two disjoint edges $e_1$ and $e_2$, to an induced $4$-cycle, $(A_1':=\varphi(A_1),A_2':=\varphi(A_2), A_3':=\varphi(A_3),A_4':=\varphi(A_4))$, generated by moving two tokens on $4$-cycle $C'$ of $G'$.
Let $\{v',w'\}$ the set of the two vertices of $C'$ that cannot be adjacent in $G'$; and let $\{x', y'\}$ be the set of the other two vertices of $C'$.
In what follows let
\[X:=A_1\cap A_2 \cap A_3 \cap A_4 \textrm{ and }
X':=A_1'\cap A_2' \cap A_3' \cap A_4'.\]
Without loss of generality, assume that
\[A_1'=X'\cup \{x',y'\}, A_2'=X'\cup \{w',y'\}, A_3'=X'\cup \{w',v'\} \textrm{ and } A_4'=X'\cup \{v',y'\}.\]
Furthermore, let \[A_5':=X' \cup \{x',v'\} \textrm{ and } A_6':=X' \cup \{x',w'\}.\]
See Figure~\ref{fig:cycles}(right).

In this paper we show the following structural results about $G$, $G'$ and $\varphi$.

\begin{restatable}{theorem}{mainlabel}\label{thm:G}
Let $\{x,v\}$ and $\{y,w\}$ the endpoints of $e_1$ and $e_2$, respectively. These vertices can be chosen so that the following hold.
\begin{enumerate}[label=\textup{(\Roman*)},leftmargin=2.6em]
\item\label{m:xy} $x$ is adjacent to $y$ in $G$, if and only if, $x'$ is adjacent to $y'$ in $G'$;

\item\label{m:vw}  $xw,yv \in E(G)$;

\item\label{m:CN} $N_G(x)\setminus \{y\}=N_G(y)\setminus\{x\}$
and $N_{G'}(x')\setminus \{y'\}=N_{G'}(y')\setminus\{x'\}$; and

\item\label{m:comp} $v,w$ lie in different components of $G \setminus \{x,y\}$,
and $v',w'$ lie in different components of $G' \setminus \{x',y'\}$.
\end{enumerate}

\end{restatable}

In what follows let $x,y,v,w,x',y',v'$ and $w'$ be as provided by Theorem~\ref{thm:G}.
Let $G_1,\dots,G_r$ be the connected components of $G \setminus \{x,y\}$ and let
$G_1',\dots,G_{r'}'$ be the connected components of $G' \setminus \{x',y'\}$.
Without loss of generality assume that $v \in G_1, w \in G_2, v' \in G_1'$ and $w' \in G_2'$.

Let
\[V_{x,1}:=\left \{ A \in V(F_k(G)): x \in A \textrm{ and }
|A \cap G_i|=|A_5 \cap G_i| \textrm{ for all } 1 \le i \le r \right \},\]
\[V_{y,1}:=\left \{ A \in V(F_k(G)): y \in A \textrm{ and }
|A \cap G_i|=|A_4 \cap G_i| \textrm{ for all } 1 \le i \le r \right \},\]
\[V_{x,2}:=\left \{ A \in V(F_k(G)): x \in A \textrm{ and }
|A \cap G_i|=|A_2 \cap G_i| \textrm{ for all } 1 \le i \le r \right \},\]
\[V_{y,2}:=\left \{ A \in V(F_k(G)): y \in A \textrm{ and }
|A \cap G_i|=|A_6 \cap G_i| \textrm{ for all } 1 \le i \le r \right \}.\]
Likewise, let
\[V_{x',1}':=\left \{ A \in V(F_{k'}(G')): x' \in A \textrm{ and }
|A \cap G_i'|=|A_5' \cap G_i'| \textrm{ for all } 1 \le i \le r' \right \},\]
\[V_{y',1}':=\left \{ A \in V(F_{k'}(G')): y' \in A \textrm{ and }
|A \cap G_i'|=|A_4' \cap G_i'| \textrm{ for all } 1 \le i \le r' \right \},\]
\[V_{x',2}':=\left \{ A \in V(F_{k'}(G')): x' \in A \textrm{ and }
|A \cap G_i'|=|A_6' \cap G_i'| \textrm{ for all } 1 \le i \le r' \right \},\]
\[V_{y',2}':=\left \{ A \in V(F_{k'}(G')): y' \in A \textrm{ and }
|A \cap G_i'|=|A_2' \cap G_i'| \textrm{ for all } 1 \le i \le r' \right \}.\]
Note that
\[A_5 \in V_{x,1}, \quad A_4 \in V_{y,1},\quad A_2 \in V_{x,2},\quad A_6 \in V_{y,2}\]
and
\[A_5' \in V_{x',1}',\quad A_4' \in V_{y',1}',\quad A_6' \in V_{x',2}',\quad A_2' \in V_{y',2}'.\]

\begin{restatable}{theorem}{tokencorr}\label{thm:token}
Either
\begin{itemize}
\item[$(a)$] \[ \varphi(V_{x,1})=V_{x',1}', \quad \varphi(V_{y,1})=V_{y',1}',
\quad \varphi(V_{x,2})=V_{y',2}', \quad \varphi(V_{y,2})=V_{x',2}';\]

\item[$(b)$] or \[ \mathfrak{c} \circ \varphi(V_{x,1})=V_{x',1}',
\quad \mathfrak{c} \circ \varphi(V_{y,1})=V_{y',1}',
\quad \mathfrak{c} \circ \varphi(V_{x,2})=V_{y',2}',
\quad \mathfrak{c} \circ \varphi(V_{y,2})=V_{x',2}';\]
in this case $k'=|G'|/2.$
\end{itemize}
\end{restatable}

In the language of Theorem~\ref{thm:token}, $\varphi$ exhibits a
\emph{twist}: relative to $v$, $\varphi$ pairs $x$ with $x'$ and $y$ with $y'$,
yet relative to $w$ it pairs $x$ with $y'$ and $y$ with $x'$. Our third main
result is that this twist can be removed by a local modification of $\varphi$.
We say that an isomorphism between token graphs \emph{twists} a $4$-cycle generated
by two disjoint edges if it sends that cycle to one generated by moving two
tokens on a $4$-cycle.

\begin{restatable}[Untwisting]{theorem}{untwistcycle}\label{thm:untwist_cycle}
There is a twin-cut automorphism $\psi$ of a token graph of $G$ such that the
isomorphism to $F_{k'}(G')$ obtained by composing $\varphi$ with $\psi$
sends $(A_1,A_2,A_3,A_4)$ to a $4$-cycle generated by two disjoint edges, and
twists no $4$-cycle generated by two disjoint edges that $\varphi$ does not.
\end{restatable}

In fact, all such twists can be removed at once.

\begin{restatable}{theorem}{twistfree}\label{thm:twistfree}
There is a composition $\Psi$ of twin-cut automorphisms of a token graph of $G$
such that the isomorphism to $F_{k'}(G')$ obtained by composing $\varphi$ with
$\Psi$ twists no $4$-cycle generated by two disjoint edges.
\end{restatable}

The paper is organized as follows. In Section~\ref{sec:structure} we prove
Theorem~\ref{thm:G}, which extracts the twin-cut structure of $G$ and $G'$ from
the behaviour of $\varphi$. In Section~\ref{sec:untwist} we prove
Theorem~\ref{thm:token}, describing how $\varphi$ pairs the two twin cuts, and
then Theorems~\ref{thm:untwist_cycle} and~\ref{thm:twistfree}, which undo the
twist by composing $\varphi$ with twin-cut flips.

\medskip\noindent\textbf{Declaration of AI use.}
The strategy of this paper, together with the proofs of parts~\ref{m:xy},
\ref{m:vw} and~\ref{m:CN} of Theorem~\ref{thm:G}, is due to the authors and
predates our use of artificial intelligence. Part~\ref{m:comp} of
Theorem~\ref{thm:G} resisted our efforts for a considerable time; it was first
proved by Anthropic's Claude Fable 5, after which we reorganized and verified the
argument and found it correct. We used OpenAI's ChatGPT~5.5 to help write the
case analysis in the proof of Theorem~\ref{thm:G} and to find the symmetries
that reduce its cases, and Anthropic's Claude Opus 4.8 to help develop and write the
results of Section~\ref{sec:untwist}. Every argument was checked in detail by the
authors, who take full responsibility for the correctness of the paper.

\section{Proof of Theorem~\ref{thm:G}}\label{sec:structure}

\subsection{Counting common neighbours in $F_k(G)$}

For every $a,b \in V(G)$, let  \[\CN(a,b):=N_G(a)\cap N_G(b),\]
That is, $\CN(a,b)$ be the set of common neighbours
of $a$ and $b$ in $G$. The following result allow us to count the common neighbors of
vertices in $F_k(G)$.

\begin{lemma}\label{lem:cn}
Let $A,B\in V(F_k(G))$.
\begin{enumerate}[label=\textup{(\roman*)},leftmargin=2.6em]

\item\label{cn:2}
If $A\triangle B=\{a,b\}$ with $a\in A$ and $b\in B$, then
\begin{align*}
\CN(A,B) & =
\left \{(A\cap B) \cup\{c\} : c\in \CN(a,b)\setminus (A\cup B)\right \} \\
& \cup
\left \{((A\cup B) \setminus\{c\})\cup\{a,b\} : c\in \CN(a,b)\cap (A\cap B)\right \};
\end{align*}
therefore,
\[\left |\CN(A,B) \right | = \left | N_G(a)\cap N_G(b)\setminus\{a,b\}\right |.\]

\item\label{cn:4}
If $A\triangle B=\{a_1,a_2,b_1,b_2\}$ with $a_1,a_2\in A$ and $b_1,b_2\in B$, then
\[
|\CN(A,B)| =
2 [a_1b_1,a_2b_2\in E(G)] + 2 [a_1b_2,a_2b_1\in E(G)],
\]
where $[\cdot]$ is $1$ if both the stated edges are present and $0$ otherwise;
therefore,
\[|\CN(A,B)| \in \{0,2,4\}.\]

\item\label{cn:6}
If $|A\triangle B|\ge 6$ then $\CN(A,B)=\emptyset$.
\end{enumerate}
\end{lemma}

\begin{proof}
Let $C \in \CN(A,B)$. Let $\{a',c_A\}:=A \triangle C$, so
 that $a' \in A$ and $c_A \in C$; likewise, let $\{b',c_B\}:=B \triangle C$, so
 that $b' \in B$ and $c_B \in C$. In particular note that $a'$ is adjacent to $c_A$,
 and $b'$ is adjacent to $c_B$.

 \begin{itemize}
  \item[\ref{cn:2}]
  First note that
\[\left \{(A\cap B) \cup\{c\} : c\in \CN(a,b)\setminus (A\cup B)\right \} \subseteq \CN(A,B)\]
and
\[\left \{((A\cup B) \setminus\{c\})\cup\{a,b\} : c\in \CN(a,b)\cap (A\cap B)\right \}
\subseteq \CN(A,B). \]
 Since $A\triangle B=\{a,b\}$, we have that $a'=a$ or $a \in A\cap B$.

 Suppose that $a'=a$. We have that $c_A\notin B$; otherwise, $C$ would be equal to $B$. Since $C$ is adjacent
  to $B$, we have that $C$ is obtained from $B$ by moving the token at $b$ to $c:=c_A=c_B$.
  This implies that $b$ is adjacent to $c$ and
 \[C \in \left \{(A\cap B) \cup\{c\} : c\in \CN(a,b)\setminus (A\cup B)\right \}.\]

 Suppose that  $a' \in A \cap B$. We have that $c_A=b$; otherwise, $\{C \triangle B\}=\{a,c_A,b,a'\}$, and $B$ would not be adjacent to $C$.
 This implies that $C$ is obtained from $B$ by moving the token at $a'$ to $a$. Thus, $a$ is adjacent to $a'$. Since
 $a'$ is adjacent to $c_A=b$. We have that
 \[C \in \left \{((A\cup B) \setminus\{c\})\cup\{a,b\} : c\in \CN(a,b)\cap (A\cap B)\right \}.\]

 \item[\ref{cn:4}]
 We have that $a' \in \{a_1,a_2\}$ and $c_A \in\{b_1,b_2\}$; otherwise $C$ would not be adjacent to $B$.
 Suppose that $a'=a_1$ and $c_a=b_1$. Since $C$ is adjacent to $B$ we have that $C$ is obtained
 from $B$ by moving the token at $b_2$ to $a_2$; thus,  $a_2$ is adjacent to $b_2$,
 and $(A \cap B) \cup \{a_1,b_2\} \in \CN (A,B)$.
 Similarly, if $a'=a_2$ and $c_a=b_2$, we have that $a_1$ is adjacent to $b_1$ and that
 $(A \cap B) \cup \{a_2,b_1\} \in \CN (A,B)$.

Suppose that $a'=a_1$ and $c_a=b_2$. Since $C$ is adjacent to $B$ we have that $C$ is obtained
 from $B$ by moving the token at $b_1$ to $a_2$; thus,  $a_2$ is adjacent to $b_1$,
 and $(A \cap B) \cup \{a_1,b_1\} \in \CN (A,B)$.
 Similarly, if $a'=a_2$ and $c_a=b_1$, we have that $a_1$ is adjacent to $b_2$ and that
 $(A \cap B) \cup \{a_2,b_2\} \in \CN (A,B)$.

 \item[\ref{cn:6}]
 Note that since $|A \triangle B| \ge 6$ and $C$ is adjacent to $B$ we have that
 $|C \triangle B| \ge 4$. This is a contradiction to the assumption that $B$ and $C$ are adjacent.
 Therefore, $A$ and $B$ do not have common neighbours.
 \end{itemize}
\end{proof}

%
%
%
%
\subsection{Proofs of~\ref{m:xy} and~\ref{m:vw}  of Theorem~\ref{thm:G}}
  Let $\{u_1,z_1\}:=e_1$ and $\{u_2,z_2\}:=e_2$, so that
 \begin{itemize}
  \item $\varphi(X \cup \{u_1,u_2\})=A_1'=X'\cup \{x',y'\}$,
  \item $\varphi(X \cup \{u_1,z_2\})=A_2'=X'\cup \{y',w'\}$,
  \item $\varphi(X \cup \{z_1,z_2\})=A_3'=X'\cup \{v',w'\}$, and
  \item $\varphi(X \cup \{z_1,u_2\})=A_4'=X'\cup \{y',v'\}$.
 \end{itemize}

 Since, $\CN(A_1',A_3')=\{A_2',A_4',A_5',A_6'\}$ and $\varphi^{-1}$ is an isomorphism, we have
 that \[|\CN(X \cup \{u_1,u_2\},X \cup \{z_1,z_2\})|=4;\] by \ref{cn:4} of Lemma~ \ref{lem:cn},
 \begin{itemize}
  \item $u_1z_2,u_2z_1 \in E(G)$, and
  \item $\{\varphi(X \cup \{u_1,z_1\}), \varphi(X \cup \{u_2,z_2\})\}=\{A_5',A_6'\}$.
 \end{itemize}



\begin{itemize}
 \item If $\varphi(X \cup \{u_1, z_1\})=X'\cup \{x', v'\}$ and
 $\varphi(X \cup \{u_2, z_2\})=X'\cup \{x', w'\}$, then let
 \[\varphi^\ast:=\varphi, \quad k^\ast:=k, \quad x:=u_1, \quad v:=z_1, \quad y:=u_2, \quad w:=z_2\]
 and 
 \[X^\ast:=X.\]
 
 \item If $\varphi(X \cup \{u_1, v_1\})=X'\cup \{x', w'\}$ and
  $\varphi(X \cup \{u_2, v_2\})=X'\cup \{x', v'\}$, then let
 \[\varphi^\ast=\varphi \circ \mathfrak{c}, \quad k^\ast=|G|-k, \quad x:=z_1, \quad v:=u_1, \quad y:=z_2, \quad w:=u_2\]
 and
 \[ X^\ast:=V(G)\setminus( X \cup \{u_1, u_2, v_1, v_2\} ).\]
\end{itemize}
Note that $\varphi^\ast$ is an isomorphism from $F_{k^\ast}(G)$ to $F_{k'}(G')$. Since $u_1z_2,u_2z_1 \in E(G)$, in either case we have

\begin{quote}
 \ref{m:vw} $xw,yv \in E(G)$;
\end{quote}

Let 
\begin{equation}\label{eq:As}
\begin{aligned}
A_1  &:= X^* \cup \{x,y\},      &\quad
A_2  &:= X^* \cup \{x,w\},      &\quad
A_3  &:= X^* \cup \{v,w\},\\
A_4  &:= X^* \cup \{v,y\},      &\quad
A_5  &:= X^* \cup \{x,v\},      &\quad
A_6  &:= X^* \cup \{y,w\}.
\end{aligned}
\end{equation}

\begin{figure}
	\centering
	\includegraphics[width=0.9\textwidth]{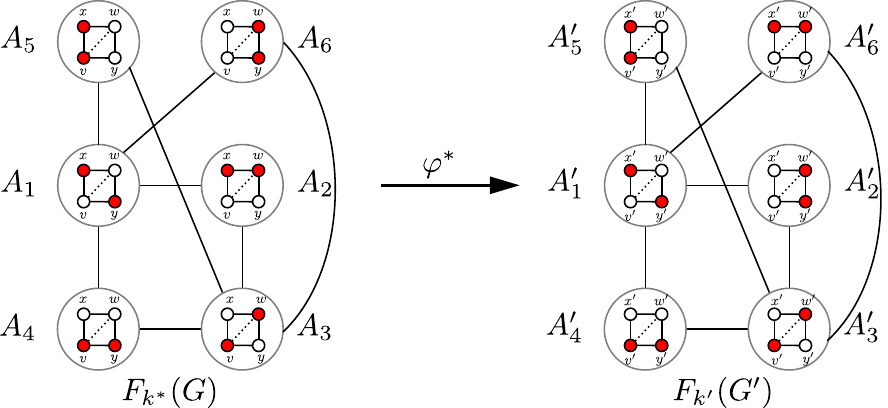}
	\caption{The dashed edges are forbidden}\label{fig:cycles}
\end{figure}

Note that
\[\varphi^\ast(A_i)=A_i'\]
for every $i=1,\dots,6.$
Since $A_5'$ is not adjacent to $A_2'$, $A_5$ is not adjacent to $A_2$. This implies that
\[v \not \sim w.\]

We have the situation depicted in Figure~\ref{fig:cycles}. In particular
$A_5$ is adjacent to $A_4$ if and only if $x$ is adjacent to $y$ in $G$; 
and $A_5'$ is adjacent to $A_4'$ if and only if $x'$ is adjacent to $y'$ in $G$.
Since $A_5$ is adjacent to $A_4$ if and only if $A_5'$ is adjacent to $A_4'$, we have 

\begin{quote}
\ref{m:xy}  \emph{$x$ is adjacent to $y$ in $G$, if and only if, $x'$ is adjacent to $y'$ in $G'$.}
\end{quote}


\subsection{Symmetries of the configuration}

In the course of proving~\ref{m:xy} and~\ref{m:vw} we have uncovered various
structural properties of $G$, $G'$ and $\varphi^\ast$. These come with
symmetries: we could swap the labels of $x$ and $y$ (and the corresponding
labels of the $A_i$) and the argument would still hold; we could invert
$\varphi^\ast$ and repeat the arguments on the primed side; and we could pass to
complements, replacing $k^\ast$ by $|G|-k^\ast$. This is not accidental --- we
now formalize these three symmetries, since they let us significantly reduce the
case analysis in what follows.

We call the data
\[
  \bigl(\varphi^\ast;\ G,k^\ast;\ G',k';\ x,y,v,w;\ x',y',v',w';\ X^\ast,X'\bigr)
\]
fixed above the \emph{standing configuration}, where $\varphi^\ast\colon
F_{k^\ast}(G)\to F_{k'}(G')$ and the tuples $A_1,\dots,A_6$ and
$A_1',\dots,A_6'$ of~\eqref{eq:As} and its primed analogue satisfy
$\varphi^\ast(A_i)=A_i'$ for $i=1,\dots,6$.

\begin{lemma}\label{lem:symmetries}
Each of the following assignments yields an isomorphism $\widehat{\varphi^\ast}$, vertices
$\widehat x,\widehat y,\widehat v,\widehat w$,
$\widehat x',\widehat y',\widehat v',\widehat w'$, sets
$\widehat X,\widehat X'$, and tuples $\widehat A_i,\widehat A_i'$ that obey the
formulas of~\eqref{eq:As} and its primed analogue in the hatted vertices, and
that satisfy $\widehat{\varphi^\ast}(\widehat A_i)=\widehat A_i'$ for $i=1,\dots,6$.

\begin{enumerate}[label=\textup{(\roman*)},leftmargin=2.6em]

\item\label{sym:inv}\emph{(Inversion.)}
  With $\sigma=(2\,6)$, set $\widehat{\varphi^\ast}:=(\varphi^\ast)^{-1}$,
  \[
    \widehat A_i:=A'_{\sigma(i)},\qquad
    \widehat A_i':=A_{\sigma(i)},
  \]
  \[
    (\widehat x,\widehat y,\widehat v,\widehat w,\widehat X)
      :=(x',y',v',w',X'),\qquad
    (\widehat x',\widehat y',\widehat v',\widehat w',\widehat X')
      :=(x,y,v,w,X^\ast).
  \]
  Here $\widehat{\varphi^\ast}\colon F_{k'}(G')\to F_{k^\ast}(G)$, and the roles of the primed
  and unprimed data are interchanged.

\item\label{sym:swap}\emph{($x$--$y$ swap.)}
  With $\tau=(2\,6)(4\,5)$, set $\widehat{\varphi^\ast}:=\varphi^\ast$,
  \[
    \widehat A_i:=A_{\tau(i)},\qquad
    \widehat A_i':=A'_{\tau(i)},
  \]
  \[
    (\widehat x,\widehat y,\widehat v,\widehat w,\widehat X)
      :=(y,x,v,w,X^\ast),\qquad
    (\widehat x',\widehat y',\widehat v',\widehat w',\widehat X')
      :=(y',x',v',w',X').
  \]

\item\label{sym:comp}\emph{(Complementation.)}
  With $\pi=(1\,3)(2\,5)(4\,6)$, set
  $\widehat{\varphi^\ast}:=\mathfrak{c}\circ\varphi^\ast\circ\mathfrak{c}$,
  \[
    \widehat A_i:=\mathfrak{c}\bigl(A_{\pi(i)}\bigr),\qquad
    \widehat A_i':=\mathfrak{c}\bigl(A'_{\pi(i)}\bigr),
  \]
  \[
    (\widehat x,\widehat y,\widehat v,\widehat w):=(y,x,v,w),\qquad
    (\widehat x',\widehat y',\widehat v',\widehat w'):=(x',y',v',w'),
  \]
  \[
    \widehat X:=V(G)\setminus\bigl(X^\ast\cup\{x,y,v,w\}\bigr),\qquad
    \widehat X':=V(G')\setminus\bigl(X'\cup\{x',y',v',w'\}\bigr).
  \]
  Here $\widehat{\varphi^\ast}\colon F_{\,|G|-k^\ast}(G)\to F_{\,|G'|-k'}(G')$.
\end{enumerate}
\end{lemma}

\begin{proof}
In each case $\widehat{\varphi^\ast}$ is an isomorphism between the stated token graphs, and
$\widehat{\varphi^\ast}(\widehat A_i)=\widehat A_i'$ follows from $\varphi^\ast(A_j)=A_j'$: for
\ref{sym:inv} it reads $(\varphi^\ast)^{-1}(A'_{\sigma(i)})=A_{\sigma(i)}$; for
\ref{sym:swap}, $\varphi^\ast(A_{\tau(i)})=A'_{\tau(i)}$; and for \ref{sym:comp},
\[
  \widehat{\varphi^\ast}\bigl(\mathfrak{c}(A_{\pi(i)})\bigr)
    =\mathfrak{c}\bigl(\varphi^\ast(A_{\pi(i)})\bigr)
    =\mathfrak{c}\bigl(A'_{\pi(i)}\bigr).
\]
Substituting the definitions of the $A_i$ and $A_i'$ shows that
$\widehat A_i,\widehat A_i'$ match the template in the hatted vertices; we check
the two moved indices of \ref{sym:inv}, the rest being identical or analogous.
For $\sigma=(2\,6)$,
\[
  \widehat A_2=A_6'=X'\cup\{x',w'\}=\widehat X\cup\{\widehat x,\widehat w\},
  \qquad
  \widehat A_6=A_2'=X'\cup\{y',w'\}=\widehat X\cup\{\widehat y,\widehat w\},
\]
while $\widehat A_i=A_i'$ agrees with the template for $i\in\{1,3,4,5\}$.
\end{proof}

Each assignment sends the standing configuration to data satisfying the same
relations, and the adjacencies among $\widehat x,\widehat y,\widehat v,\widehat w$
(and among their primed images) match those among $x,y,v,w$; in particular
$\widehat x\widehat v,\widehat y\widehat w,\widehat x\widehat w,\widehat y\widehat v$
are edges, $\widehat v\widehat w$ is a non-edge, and $\widehat x\widehat y$ is an
edge if and only if $xy$ is. Consequently, any conclusion drawn from the
standing configuration together with whichever of~\ref{m:xy}--\ref{m:comp} are
established at the point of use transfers verbatim to the hatted data. We use
this to apply an argument to $(\varphi^\ast)^{-1}$ with the roles of $G$ and $G'$
interchanged~(\ref{sym:inv}); to interchange $x$ with $y$ and $x'$ with
$y'$~(\ref{sym:swap}); and to replace $k^\ast$ by $|G|-k^\ast$ and pass to
complements~(\ref{sym:comp}).

\subsection{Proof of~\ref{m:CN} of Theorem~\ref{thm:G}}

By the inversion symmetry~\ref{sym:inv} of Lemma~\ref{lem:symmetries}, which
interchanges the primed and unprimed data, it suffices to prove that
\[N_{G'}(x')\setminus\{y'\}=N_{G'}(y')\setminus \{x'\};\]
the statement $N_{G}(x)\setminus\{y\}=N_{G}(y)\setminus \{x\}$ then follows by
applying the same argument to $(\varphi^\ast)^{-1}$.

For a contradiction, suppose that
\[
  N_{G'}(x')\setminus\{y'\}
  \ne
  N_{G'}(y')\setminus\{x'\}.
\]
Choose
\[
  b'\in
  \bigl(N_{G'}(x')\setminus\{y'\}\bigr)
  \triangle
  \bigl(N_{G'}(y')\setminus\{x'\}\bigr).
\]
If
\[
  b'\in
  N_{G'}(y')\setminus\bigl(N_{G'}(x')\cup\{x'\}\bigr),
\]
we apply the $x$--$y$ swap~\ref{sym:swap} of Lemma~\ref{lem:symmetries},
which interchanges $x'$ with $y'$ (and $x$ with $y$) while preserving the
standing configuration and the already-established
statements~\ref{m:xy} and~\ref{m:vw}; this turns the present case into the
other one. Thus, after this swap if necessary, we may assume that
\begin{equation*}
  b'\in
  N_{G'}(x')
  \setminus
  \bigl(N_{G'}(y')\cup\{y'\}\bigr).
\end{equation*}


We may further assume that
\begin{equation}\label{eq:bnotX}
  b'\notin X'.
\end{equation}
Indeed, suppose that $b'\in X'$. We apply the complementation~\ref{sym:comp}
of Lemma~\ref{lem:symmetries}, which replaces $k^\ast$ by $|G|-k^\ast$ and
passes to complements via
$\mathfrak{c}\circ\varphi^\ast\circ\mathfrak{c}$, while preserving the
standing configuration and the already-established statements~\ref{m:xy}
and~\ref{m:vw}. Its new set is
\[
  \widehat X'=V(G')\setminus\bigl(X'\cup\{x',y',v',w'\}\bigr),
\]
so $b'\notin\widehat X'$, since $b'\in X'$; and since the primed vertices
$x',y',v',w'$ are left unchanged,
\[
  b'\in
  N_{G'}(x')
  \setminus\bigl(N_{G'}(y')\cup\{y'\}\bigr)
\]
still holds. Replacing the original data by the complemented data, we may
assume~\eqref{eq:bnotX}.


Let
 \[B_1':=\{b',y'\} \cup X', \quad B_2':=\{b', w'\} \cup X', \quad B_3':=\{b',x'\} \cup X' \textrm{ and } B_4':=\{b',v'\} \cup X';\]
and let
 \[B_i:={\varphi^\ast}^{-1}(B_i')\]
 for $i=1,\dots,4$. We have the situation depicted in Figure~\ref{fig:second-cycle}.

\begin{figure}
  \centering
  \includegraphics[width=0.7\textwidth]{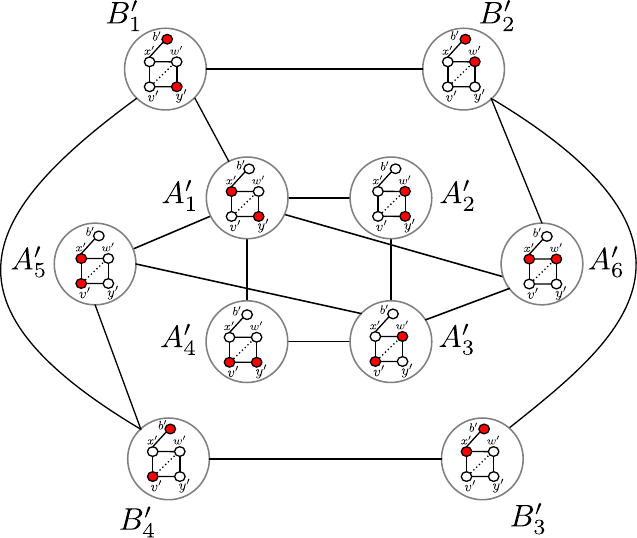}
  \caption{Secondary $4$-cycle}
  \label{fig:second-cycle}
\end{figure}
Note that $B_3'\not\sim A_1'$, since
$B_3'\triangle A_1'=\{b',y'\}$ and $b'\not\sim y'$. Moreover,
\[
  |B_3'\triangle A_i'|=4
  \qquad\text{for }i=2,3,4.
\]
Therefore,
\begin{equation}
		\label{eq:vertex-B3}
		\textrm{ $B_3'$ is not adjacent to any of $A_1',A_2',A_3',A_4'$.}
	\end{equation}

Let $a, b \in V(G)$ such that $B_1$ is obtained from $A_1$ by moving a token from $a$ to $b$.
\begin{itemize}
  \item \emph{Suppose that $a = x$.}

    Thus, \[B_1 = (A_1 \setminus \{x\}) \cup \{b\},\] with $b \in V(G) \setminus \{x, y, v, w\}$.

    Note that $B_2 \in \CN(B_1,A_6)$ and $B_1 \triangle A_6=\{b,w\}$;
    By \ref{cn:2} of Lemma~\ref{lem:cn}, we have
    that there exists $u\in \CN(W,b)$ such that either
    \[B_2 = (A_6 \setminus \{w\}) \cup \{u\} \textrm{ or } B_2=(A_6\setminus \{u\}) \cup \{b\}. \]

  Note that $B_4 \in \CN(B_1,A_5)$, $B_1 \triangle A_5=\{b,y,x,v\}$, $A_1=(B_1 \setminus \{b\}) \cup \{x\}$ and $A_4=(B_1 \setminus \{b\}) \cup \{v\}$; by \ref{cn:4} of Lemma~\ref{lem:cn}, we have
  that \[B_4 = (B_1 \setminus \{y\}) \cup \{v\} \textrm{ or }B_4 = (B_1 \setminus \{y\}) \cup \{x\}.\]

  \begin{itemize}
  \item \emph{Suppose that $B_2 = (A_6 \setminus \{w\}) \cup \{u\}$.}

  \begin{itemize}
   \setlength{\itemsep}{1em}
    \item\emph{Suppose that $B_4 = (B_1 \setminus \{y\}) \cup \{v\}$.}

    Note that $B_3 \in \CN(B_2,B_4)$,  $B_2 \triangle B_4 = \{u,y,b,v\}$, $B_1= (B_2 \setminus \{u\}) \cup \{b\}$
    $A_4=(B_2 \setminus \{u\}) \cup \{v\}$ and $A_4\not \sim B_4$. By \ref{cn:4} of Lemma~\ref{lem:cn}, we have that
      \[B_3=(B_2 \setminus \{y\}) \cup \{v\}.\]
    Therefore, $B_3$ is adjacent to $A_3$, contradicting~\eqref{eq:vertex-B3}.

    \item \emph{Suppose that $B_4 = (B_1 \setminus \{y\}) \cup \{x\}$.}

    Note that $B_3 \in CN(B_2,B_4)$, $B_2 \triangle B_4 = \{u,y,b,x\}$, $B_1= (B_2 \setminus \{u\}) \cup \{b\}$
    $A_1=(B_2 \setminus \{u\}) \cup \{x\}$ and $A_1\not \sim B_1$; by \ref{cn:4} of Lemma~\ref{lem:cn}, we have
    that
    \[B_3=(B_2 \setminus \{y\}) \cup \{x\}.\]
    Therefore, $B_3$ is adjacent to $A_2$, contradicting~\eqref{eq:vertex-B3}.
  \end{itemize}

  \item  \emph{Suppose that $B_2=(A_6\setminus \{u\}) \cup \{b\}$.}

  \begin{itemize}
   \setlength{\itemsep}{1em}

   \item \emph{Suppose that $B_4 = (B_1 \setminus \{y\}) \cup \{v\}$.}

    Note that $B_3 \in \CN(B_2,B_4)$, $B_2 \triangle B_4=\{w,y,v,u\}$, and $v\not\sim w$. By \ref{cn:4} of Lemma~\ref{lem:cn}, we have that $B_3$ is equal to one of
    \[
      (B_2 \setminus \{w\}) \cup \{u\} \quad \textrm{ or }
      (B_2 \setminus \{y\}) \cup \{v\}
    \]
    However, $A_1$ is obtained from $(B_2 \setminus \{w\}) \cup \{u\} $ by moving the token at
    $b$ to $x$, and $A_3$ is obtained from $(B_2 \setminus \{y\}) \cup \{v\}$ by moving the token
    at $b$ to $u$. In both cases we contradict~\eqref{eq:vertex-B3}.

  \item \emph{Suppose that $B_4 = (B_1 \setminus \{y\}) \cup \{x\}$.}

    Note that $B_3 \in \CN(B_2,B_4)$ and $B_2 \triangle B_4=\{w,y,x,u\}$. By \ref{cn:4} of Lemma~\ref{lem:cn}, we
    have that $B_3$ is equal to one of
    \[
      (B_2 \setminus \{w\}) \cup \{x\}, \quad
      (B_2 \setminus \{w\}) \cup \{u\}, \quad
      (B_2 \setminus \{y\}) \cup \{x\}, \quad
      (B_2 \setminus \{y\}) \cup \{u\}.
    \]
    However, the first two vertices are adjacent to $A_1$, and the last two vertices are adjacent to $A_2$.
    In each case, this contradicts~\eqref{eq:vertex-B3}.
  \end{itemize}
\end{itemize}

  \item \emph{Suppose that $a=y$.}

    Thus,
    \[
      B_1=(A_1\setminus\{y\})\cup\{b\},
    \]
    where $b\in V(G)\setminus\{x,y,v,w\}$; in particular, $b\sim y$.

    Note that $B_2\in\CN(B_1,A_6)$,
    $B_1\triangle A_6=\{x,b,y,w\}$,
    $A_1=(B_1\setminus\{b\})\cup\{y\}$, and
    $A_2=(B_1\setminus\{b\})\cup\{w\}$.
    Since $B_2\notin\{A_1,A_2\}$, by \ref{cn:4} of
    Lemma~\ref{lem:cn},
    \[
      B_2=(B_1\setminus\{x\})\cup\{w\}
      \quad\text{or}\quad
      B_2=(B_1\setminus\{x\})\cup\{y\}.
    \]
    Observe also that \[x\not\sim b.\] Indeed, in the first alternative
    $A_2\not\sim B_2$, since $A_2'\not\sim B_2'$, whereas in the
    second alternative $A_1\not\sim B_2$, since
    $A_1'\not\sim B_2'$.

    Note also that $B_4\in\CN(B_1,A_5)$ and
    $B_1\triangle A_5=\{b,v\}$. By \ref{cn:2} of
    Lemma~\ref{lem:cn}, there exists $u\in\CN(v,b)$ such that
    \[
      B_4=(A_5\setminus\{v\})\cup\{u\}
      \quad\text{or}\quad
      B_4=(A_5\setminus\{u\})\cup\{b\}.
    \]

    \begin{itemize}
    \setlength{\itemsep}{1em}
      \item \emph{Suppose that
      $B_4=(A_5\setminus\{v\})\cup\{u\}$.}

        Since $B_4\notin\{A_1,A_2\}$, we have $u\notin\{y,w\}$.
        \begin{itemize}
          \item \emph{Suppose that
          $B_2=(B_1\setminus\{x\})\cup\{w\}$.}

            Note that $B_1,B_3\in\CN(B_2,B_4)$,
            $B_2\triangle B_4=\{b,w,x,u\}$,
            $B_1=(B_4\setminus\{u\})\cup\{b\}$, $A_2=(B_4\setminus\{u\})\cup\{w\}$.
            Since $x\not\sim b$, by \ref{cn:4} of
            Lemma~\ref{lem:cn},
            \[
              B_3=(B_4\setminus\{x\})\cup\{w\}.
            \]
            Since $u\sim v$, we have $B_3\sim A_3$,
            contradicting~\eqref{eq:vertex-B3}.

          \item \emph{Suppose that
          $B_2=(B_1\setminus\{x\})\cup\{y\}$.}

            Note that $B_1,B_3\in\CN(B_2,B_4)$,
            $B_2\triangle B_4=\{b,y,x,u\}$,
            $B_1=(B_4\setminus\{u\})\cup\{b\}$ and
            $A_1=(B_4\setminus\{u\})\cup\{y\}$.
            Since $x\not\sim b$, by \ref{cn:4} of
            Lemma~\ref{lem:cn},
            \[
              B_3=(B_4\setminus\{x\})\cup\{y\}.
            \]
            Since $u\sim v$, we have $B_3\sim A_4$,
            contradicting~\eqref{eq:vertex-B3}.
        \end{itemize}

      \item \emph{Suppose that
      $B_4=(A_5\setminus\{u\})\cup\{b\}$.}

        In this case, we have that \[u\in X^\ast.\]
        \begin{itemize}
          \item \emph{Suppose that
          $B_2=(B_1\setminus\{x\})\cup\{w\}$.}

            Note that $B_3\in\CN(B_2,B_4)$ and
            $B_2\triangle B_4=\{x,v,u,w\}$. By \ref{cn:4} of
            Lemma~\ref{lem:cn}, $B_3$ is equal to one of
            \[
              (B_4\setminus\{v\})\cup\{w\},\quad
              (B_4\setminus\{v\})\cup\{u\},\quad
              (B_4\setminus\{x\})\cup\{w\},\quad
              (B_4\setminus\{x\})\cup\{u\}.
            \]
            These four vertices are adjacent to $A_2$, $A_1$, $A_3$,
            and $A_4$, respectively, contradicting~\eqref{eq:vertex-B3}.

          \item \emph{Suppose that
          $B_2=(B_1\setminus\{x\})\cup\{y\}$.}

            Note that $B_3\in\CN(B_2,B_4)$ and
            $B_2\triangle B_4=\{x,v,u,y\}$. By \ref{cn:4} of
            Lemma~\ref{lem:cn}, $B_3$ is equal to one of
            \[
              (B_4\setminus\{v\})\cup\{y\},\quad
              (B_4\setminus\{v\})\cup\{u\},\quad
              (B_4\setminus\{x\})\cup\{y\},\quad
              (B_4\setminus\{x\})\cup\{u\}.
            \]
            The first two vertices are adjacent to $A_1$, and the
            last two are adjacent to $A_4$. In every case, this
            contradicts~\eqref{eq:vertex-B3}.
        \end{itemize}
    \end{itemize}

  \item \emph{Suppose that $a\notin\{x,y\}$.}

    Thus, $a\in X^\ast$ and
    \[
      B_1=(A_1\setminus\{a\})\cup\{b\},
    \]
    where $b\notin X^\ast\cup\{x,y\}$; in particular, $a\sim b$.
    Since  $B_3'\notin\{A_1',A_2',A_3',A_4'\}$, we have that
    \[B_3\notin\{A_1,A_2,A_3,A_4\}.\]
    We divide the proof
    according to whether $b=v$, $b=w$, or $b\notin\{v,w\}$.

    \begin{itemize}
    \setlength{\itemsep}{1em}
      \item \emph{Suppose that $b=v$.}

        In this case,
        \[
          B_1=(A_1\setminus\{a\})\cup\{v\}.
        \]
        Note that $B_2\in\CN(B_1,A_6)$,
        $B_1\triangle A_6=\{x,v,a,w\}$,
        $A_1=(B_1\setminus\{v\})\cup\{a\}$,
        and $v\not\sim w$.
        By \ref{cn:4} of Lemma~\ref{lem:cn},
        \[
          B_2=(B_1\setminus\{x\})\cup\{w\}.
        \]

        Note also that $B_4\in\CN(B_1,A_5)$ and
        $B_1\triangle A_5=\{y,a\}$. By \ref{cn:2} of
        Lemma~\ref{lem:cn}, there exists $u\in\CN(y,a)$ such that
        \[
          B_4=(B_1\setminus\{y\})\cup\{u\}
          \quad\text{or}\quad
          B_4=(B_1\setminus\{u\})\cup\{a\}.
        \]

        \begin{itemize}
          \item \emph{Suppose that
          $B_4=(B_1\setminus\{y\})\cup\{u\}$.}

            Since $B_4'\not\sim A_2'$, we have $B_4\not\sim A_2$,
            and hence $u\ne w$.
            Note that $B_3\in\CN(B_2,B_4)$ and
            $B_2\triangle B_4=\{w,y,x,u\}$. By \ref{cn:4} of
            Lemma~\ref{lem:cn}, $B_3$ is equal to one of
            \[
              (B_4\setminus\{u\})\cup\{w\},\quad
              (B_4\setminus\{u\})\cup\{y\},\quad
              (B_4\setminus\{x\})\cup\{w\},\quad
              (B_4\setminus\{x\})\cup\{y\}.
            \]
            These four vertices are adjacent to $A_2$, $A_1$, $A_3$,
            and $A_4$, respectively, contradicting~\eqref{eq:vertex-B3}.

          \item \emph{Suppose that
          $B_4=(B_1\setminus\{u\})\cup\{a\}$.}

            Since $B_4\notin\{A_1,A_4,A_5\}$, we have
            $u\in X^\ast\setminus\{a\}$. Moreover,
            $A_4\not\sim B_4$, since $A_4'\not\sim B_4'$; hence
            $u\not\sim x$. Note that $B_3\in\CN(B_2,B_4)$,
            $B_2\triangle B_4=\{w,u,a,x\}$, $x\sim w$, and
            $a\sim u$. By \ref{cn:4} of Lemma~\ref{lem:cn},
            \[
              B_3=(B_2\setminus\{w\})\cup\{x\}
              \quad\text{or}\quad
              B_3=(B_2\setminus\{u\})\cup\{a\}.
            \]
            The first vertex is adjacent to $A_1$, and the second is
            adjacent to $A_3$, since $u\sim y$. In either case, this
            contradicts~\eqref{eq:vertex-B3}.
        \end{itemize}

      \item \emph{Suppose that $b=w$.}

        In this case,
        \[
          B_1=(A_1\setminus\{a\})\cup\{w\}.
        \]
        Note that $B_4\in\CN(B_1,A_5)$,
        $B_1\triangle A_5=\{y,w,a,v\}$,
        $A_1=(B_1\setminus\{w\})\cup\{a\}$,
         and $v\not\sim w$.
        By \ref{cn:4} of Lemma~\ref{lem:cn},
        \[
          B_4=(B_1\setminus\{y\})\cup\{v\}.
        \]

        Note also that $B_2\in\CN(B_1,A_6)$ and
        $B_1\triangle A_6=\{x,a\}$. By \ref{cn:2} of
        Lemma~\ref{lem:cn}, there exists $u\in\CN(x,a)$ such that
        \[
          B_2=(B_1\setminus\{x\})\cup\{u\}
          \quad\text{or}\quad
          B_2=(B_1\setminus\{u\})\cup\{a\}.
        \]

        \begin{itemize}
          \item \emph{Suppose that
          $B_2=(B_1\setminus\{x\})\cup\{u\}$.}

            Since $B_2'\not\sim A_4'$, we have $B_2\not\sim A_4$,
            and hence $u\ne v$; otherwise,
            $B_2\triangle A_4=\{w,a\}$ and $a\sim w$.
            Note that $B_3\in\CN(B_2,B_4)$ and
            $B_2\triangle B_4=\{y,u,x,v\}$. By \ref{cn:4} of
            Lemma~\ref{lem:cn}, $B_3$ is equal to one of
            \[
              (B_2\setminus\{y\})\cup\{x\},\quad
              (B_2\setminus\{y\})\cup\{v\},\quad
              (B_2\setminus\{u\})\cup\{x\},\quad
              (B_2\setminus\{u\})\cup\{v\}.
            \]
            These four vertices are adjacent to $A_2$, $A_3$, $A_1$,
            and $A_4$, respectively, contradicting~\eqref{eq:vertex-B3}.

          \item \emph{Suppose that
          $B_2=(B_1\setminus\{u\})\cup\{a\}$.}

            Since $B_2\notin\{A_1,A_2,A_6\}$, we have
            $u\in X^\ast\setminus\{a\}$. Moreover,
            $B_4\not\sim A_2$, since $B_4'\not\sim A_2'$; hence
            $a\not\sim v$. Note that $B_3\in\CN(B_2,B_4)$ and
            $B_2\triangle B_4=\{a,y,u,v\}$.
            By \ref{cn:4} of Lemma~\ref{lem:cn},
            \[
              B_3=(B_2\setminus\{y\})\cup\{v\}.
            \]
            This vertex is adjacent to $A_3$, contradicting~\eqref{eq:vertex-B3}.
        \end{itemize}

      \item \emph{Suppose that $b\notin\{v,w\}$.}

        Note that $B_4\in\CN(B_1,A_5)$,
        $B_1\triangle A_5=\{y,b,a,v\}$, and
        $A_1=(B_1\setminus\{b\})\cup\{a\}$. By \ref{cn:4} of
        Lemma~\ref{lem:cn}, $B_4$ is equal to one of
        \[
          (B_1\setminus\{y\})\cup\{v\},\quad
          (B_1\setminus\{y\})\cup\{a\},\quad
          (B_1\setminus\{b\})\cup\{v\}.
        \]
        Similarly, $B_2\in\CN(B_1,A_6)$,
        $B_1\triangle A_6=\{x,b,a,w\}$, and by \ref{cn:4} of
        Lemma~\ref{lem:cn}, $B_2$ is equal to one of
        \[
          (B_1\setminus\{x\})\cup\{w\},\quad
          (B_1\setminus\{x\})\cup\{a\},\quad
          (B_1\setminus\{b\})\cup\{w\}.
        \]

        \begin{itemize}
        \setlength{\itemsep}{1em}
          \item \emph{Suppose that
          $B_4=(B_1\setminus\{y\})\cup\{v\}$.}

            \begin{itemize}
              \item \emph{Suppose that
              $B_2=(B_1\setminus\{x\})\cup\{w\}$.}

                Note that $B_3\in\CN(B_2,B_4)$ and
                $B_2\triangle B_4=\{w,y,x,v\}$. By \ref{cn:4} of
                Lemma~\ref{lem:cn}, $B_3$ is equal to one of
                \[
                  (B_2\setminus\{w\})\cup\{x\},\quad
                  (B_2\setminus\{y\})\cup\{x\},\quad
                  (B_2\setminus\{y\})\cup\{v\}.
                \]
                These three vertices are adjacent to $A_1$, $A_2$,
                and $A_3$, respectively, contradicting~\eqref{eq:vertex-B3}.

              \item \emph{Suppose that
              $B_2=(B_1\setminus\{x\})\cup\{a\}$.}

                Since $B_2\sim A_6$ and
                $B_2\triangle A_6=\{b,w\}$, we have $b\sim w$.
                Note that $B_3\in\CN(B_2,B_4)$ and
                $B_2\triangle B_4=\{y,a,x,v\}$. By \ref{cn:4} of
                Lemma~\ref{lem:cn}, $B_3$ is equal to one of
                \[
                  (B_2\setminus\{y\})\cup\{x\},\quad
                  (B_2\setminus\{y\})\cup\{v\},\quad
                  (B_2\setminus\{a\})\cup\{x\},\quad
                  (B_2\setminus\{a\})\cup\{v\}.
                \]
                These four vertices are adjacent to $A_2$, $A_3$, $A_1$,
                and $A_4$, respectively, contradicting~\eqref{eq:vertex-B3}.

              \item \emph{Suppose that
              $B_2=(B_1\setminus\{b\})\cup\{w\}$.}

                Since $B_2\sim A_6$ and
                $B_2\triangle A_6=\{a,x\}$, we have $a\sim x$.
                Note that $B_3\in\CN(B_2,B_4)$,
                $B_2\triangle B_4=\{w,y,b,v\}$, and
                $v\not\sim w$. By \ref{cn:4} of
                Lemma~\ref{lem:cn},
                \[
                  B_3=(B_2\setminus\{y\})\cup\{v\}.
                \]
                This vertex is adjacent to $A_3$, contradicting~\eqref{eq:vertex-B3}.
            \end{itemize}

          \item \emph{Suppose that
          $B_4=(B_1\setminus\{y\})\cup\{a\}$.}

            Since $B_4\sim A_5$ and
            $B_4\triangle A_5=\{b,v\}$, we have $b\sim v$.
            Moreover, $B_4\not\sim A_2$, since $B_4'\not\sim A_2'$;
            as $B_4\triangle A_2=\{b,w\}$, we have $b\not\sim w$.

            \begin{itemize}
              \item \emph{Suppose that
              $B_2=(B_1\setminus\{x\})\cup\{w\}$.}

                Note that $B_3\in\CN(B_2,B_4)$ and
                $B_2\triangle B_4=\{w,y,a,x\}$. By \ref{cn:4} of
                Lemma~\ref{lem:cn}, $B_3$ is equal to one of
                \[
                  (B_2\setminus\{w\})\cup\{a\},\quad
                  (B_2\setminus\{w\})\cup\{x\},\quad
                  (B_2\setminus\{y\})\cup\{a\},\quad
                  (B_2\setminus\{y\})\cup\{x\}.
                \]
                These four vertices are adjacent to $A_4$, $A_1$, $A_3$,
                and $A_2$, respectively, contradicting~\eqref{eq:vertex-B3}.

              \item \emph{Suppose that
              $B_2=(B_1\setminus\{x\})\cup\{a\}$.}

                Since $B_2\sim A_6$ and
                $B_2\triangle A_6=\{b,w\}$, we have $b\sim w$,
                contradicting $b\not\sim w$.

              \item \emph{Suppose that
              $B_2=(B_1\setminus\{b\})\cup\{w\}$.}

              This contradicts $b \not \sim w$.

            \end{itemize}

          \item \emph{Suppose that
          $B_4=(B_1\setminus\{b\})\cup\{v\}$.}

            Since $B_4\sim A_5$ and
            $B_4\triangle A_5=\{a,y\}$, we have $a\sim y$.

            \begin{itemize}
              \item \emph{Suppose that
              $B_2=(B_1\setminus\{x\})\cup\{w\}$.}

                Note that $B_3\in\CN(B_2,B_4)$ and
                $B_2\triangle B_4=\{w,b,x,v\}$. By \ref{cn:4} of
                Lemma~\ref{lem:cn}, $B_3$ is equal to one of
                \[
                  (B_2\setminus\{w\})\cup\{x\},\quad
                  (B_2\setminus\{b\})\cup\{v\}.
                \]
                These two vertices are adjacent to $A_1$
                and $A_3$, respectively, contradicting~\eqref{eq:vertex-B3}.

              \item \emph{Suppose that
              $B_2=(B_1\setminus\{x\})\cup\{a\}$.}

                Note that $B_3\in\CN(B_2,B_4)$ and
                $B_2\triangle B_4=\{a,b,x,v\}$. By \ref{cn:4} of
                Lemma~\ref{lem:cn}, $B_3$ is equal to one of
                \[
                  (B_2\setminus\{a\})\cup\{x\},\quad
                  (B_2\setminus\{a\})\cup\{v\},\quad
                  (B_2\setminus\{b\})\cup\{x\},\quad
                  (B_2\setminus\{b\})\cup\{v\}.
                \]
                The first two vertices are adjacent to $A_1$ and $A_4$,
                respectively, while the last two are $A_1$ and $A_4$.
                Thus, in every case we contradict either
                $B_3\notin\{A_1,A_4\}$ or~\eqref{eq:vertex-B3}.

              \item \emph{Suppose that
              $B_2=(B_1\setminus\{b\})\cup\{w\}$.}

                Since $B_2\triangle A_2=\{a,y\}$ and $a\sim y$, we have
                $B_2\sim A_2$. However, $B_2'\not\sim A_2'$, a
                contradiction.
            \end{itemize}
        \end{itemize}
    \end{itemize}
\end{itemize}

In every case we reached a contradiction. Therefore
\[
  N_{G'}(x')\setminus\{y'\}=N_{G'}(y')\setminus\{x'\},
\]
and, by the reduction at the start of this subsection, this completes the proof
of~\ref{m:CN}.


\subsection{Proof of~\ref{m:comp} of Theorem~\ref{thm:G}}

Let
\[
  N:=N_G(x)\setminus\{y\}=N_G(y)\setminus\{x\}
  \quad\text{and}\quad
  N':=N_{G'}(x')\setminus\{y'\}=N_{G'}(y')\setminus\{x'\},
\]
where the equalities follow from~\ref{m:CN}.

Let $P,Q \in V(F_{k^\ast}(G))$. In what follows assume that
\[P \triangle  Q=\{x,y\} \textrm{ with } x\in P \textrm{ and } y \in Q.\]
In this case, we say that $(P,Q)$ is an \emph{$xy$-pair}.
Let
\[S:= P\cap Q.\]
By \ref{cn:2} of Lemma~\ref{lem:cn}, we have that
\[\CN(P,Q)  = \underbrace{\bigl\{S\cup\{c\} : c\in N\setminus S \bigr\}}_{\text{``$Zero$-type''}}
\cup \underbrace{\bigl\{ (S\setminus\{c\})\cup\{x,y\} : c\in N\cap S \bigr\}}_{\text{``$Double$-type''}}, \]
and
\begin{equation}\label{eq:comp-cnxy}
|\CN(P,Q)| =|N|.
\end{equation}
The common neighbours of $P$ and $Q$ of the form $\{S\cup\{c\} : c\in N\setminus S\}$, do not have
tokens at either $x$ or $y$---we call them \emph{$Zero$-type}. Whereas the common neighbours of $P$ and $Q$ of the form $\{ (S\setminus\{c\})\cup\{x,y\} : c\in N\cap S\}$, have a token at each of
$x$ and $y$---we call them \emph{$Double$-type}.

The same description applies on the primed side, with $N'$ in place of $N$.
In particular, applying $\varphi^\ast$ to the $xy$-pair $(A_5,A_4)$, whose image
$(A_5',A_4')$ is the $x'y'$-pair with common part $X'\cup\{v'\}$, we obtain:
\[|N|=|N'|.\]

We say that $(P,Q)$ is \emph{strong} if in addition
\[\varphi^\ast(P)\sd\varphi^\ast(Q)=\{x',y'\}\] with $x'\in\varphi^\ast(P)$
and $y'\in\varphi^\ast(Q)$.
Whenever $(P,Q)$ is a strong $xy$-pair,
we write
\[
  P':=\varphi^\ast(P),
  \qquad
  Q':=\varphi^\ast(Q),
  \qquad
  S':=P'\cap Q'.
\]
Thus
\[
  P'=S'\cup\{x'\}
  \qquad\text{and}\qquad
  Q'=S'\cup\{y'\}.
\]

Let $u\in S$ and let
\[
  z\in V(G)\setminus\bigl(S\cup\{x,y\}\bigr)
\]
be such that $uz\in E(G)$. Set
\[
  \widetilde P:=(P\setminus\{u\})\cup\{z\}
  \qquad\text{and}\qquad
  \widetilde Q:=(Q\setminus\{u\})\cup\{z\}.
\]
Thus, $\widetilde P$ and $\widetilde Q$ are obtained from $P$ and
$Q$, respectively, by moving the token at $u$ along the edge $uz$
to $z$. Note that $(\widetilde P,\widetilde Q)$ is again an
$xy$-pair. We call this operation a \emph{move from $u$ to $z$}
and denote it by
\[
  (P,Q)\xrightarrow{\,u\to z\,}
  (\widetilde P,\widetilde Q).
\]

\begin{lemma}[Step Lemma]\label{lem:comp-step}
Suppose that $(P,Q)$ is a strong $xy$-pair and that
\[
  (P,Q)\xrightarrow{\,u\to z\,}(\widetilde P,\widetilde Q).
\]
Let
\[
  \widetilde P':=\varphi^\ast(\widetilde P),
  \qquad
  \widetilde Q':=\varphi^\ast(\widetilde Q).
\]
Then there exist $u',s'\in S'$ and
$t',r'\notin S'\cup\{x',y'\}$ such that
\[
  \widetilde P'
    =(P'\setminus\{u'\})\cup\{t'\}
\textrm{ and }
  \widetilde Q'
    =(Q'\setminus\{s'\})\cup\{r'\}.
\]
In particular, we have the  \emph{weak invariant}
\[
  x'\in\widetilde P'\setminus\widetilde Q'
  \qquad\text{and}\qquad
  y'\in\widetilde Q'\setminus\widetilde P'.
\]
Moreover, the following hold.
\begin{enumerate}[label=\textup{(\roman*)},leftmargin=2.6em]
  \item If $xy\in E(G)$ or $|N|\ge 3$, then
  \[
    u'=s'
    \qquad\text{and}\qquad
    t'=r'.
  \]
  Hence $(\widetilde P,\widetilde Q)$ is strong, with common part
  \[
    \widetilde S'
      =(S'\setminus\{u'\})\cup\{t'\}.
  \]

  \item If $xy\notin E(G)$ and $|N|=2$,  then exactly one of the
  following cases occurs:
  \begin{description}[leftmargin=2.6em]
    \item[\textup{(S)}]
      $u'=s'$ and $t'=r'$. In this case,
      $(\widetilde P,\widetilde Q)$ is strong, with common part
      \[
        \widetilde S'
          =(S'\setminus\{u'\})\cup\{t'\}.
      \]

    \item[\textup{(A)}]
      $u'=s'$, $t'\ne r'$ and $ t',r'\in N'\setminus S'.$

    \item[\textup{(B)}]
      $t'=r'$, $u'\ne s'$ and  $ u',s'\in N'\cap S'$.
  \end{description}
\end{enumerate}
\end{lemma}
\begin{proof}
By the definition of a move,
\[
  \widetilde P=(P\setminus\{u\})\cup\{z\}.
\]
Thus, $\widetilde P\sim P$. Moreover,
\[
  \widetilde P\triangle Q=\{u,z,x,y\},
\]
and hence
\[
  \widetilde P\ne Q
  \qquad\text{and}\qquad
  \widetilde P\not\sim Q.
\]
Since $\varphi^\ast$ is an isomorphism, it follows that
\[
  \widetilde P'\sim P',\qquad
  \widetilde P'\ne Q',
  \qquad\text{and}\qquad
  \widetilde P'\not\sim Q'.
\]

Since $\widetilde P'\sim P'$, there exist
\[
  u'\in S'\cup\{x'\}
  \qquad\text{and}\qquad
  t'\notin S'\cup\{x'\}
\]
such that $u't'\in E(G')$ and
\[
  \widetilde P'=(P'\setminus\{u'\})\cup\{t'\}.
\]
We claim that
\[
  u'\in S'
  \qquad\text{and}\qquad
  t'\ne y'.
\]

Suppose first that $u'=x'$. If $t'=y'$, then
\[
  \widetilde P'=Q',
\]
a contradiction. Otherwise, $t'\ne y'$. Since $x't'\in E(G')$,
it follows from~\ref{m:CN} that $y't'\in E(G')$. Therefore,
\[
  \widetilde P'=S'\cup\{t'\}
  \sim S'\cup\{y'\}=Q',
\]
again a contradiction. Hence $u'\in S'$.

Now suppose that $t'=y'$. Since $u'y'\in E(G')$ and
$u'\in S'$, it follows from~\ref{m:CN} that $u'x'\in E(G')$.
Consequently,
\[
  \widetilde P'
    =(S'\setminus\{u'\})\cup\{x',y'\}
    \sim S'\cup\{y'\}
    =Q',
\]
a contradiction. Thus, $t'\ne y'$. Since we already know that
$t'\notin S'\cup\{x'\}$, we conclude that
\[
  u'\in S'
  \qquad\text{and}\qquad
  t'\notin S'\cup\{x',y'\}.
\]

Applying the same argument with the roles of $P$ and $Q$, and of
$x'$ and $y'$, interchanged, there exist
\[
  s'\in S'
  \qquad\text{and}\qquad
  r'\notin S'\cup\{x',y'\}
\]
such that
\[
  \widetilde Q'=(Q'\setminus\{s'\})\cup\{r'\}.
\]
Thus,
\[
  \widetilde P'
    =(S'\setminus\{u'\})\cup\{x',t'\}
\]
and
\[
  \widetilde Q'
    =(S'\setminus\{s'\})\cup\{y',r'\}.
\]
In particular,
\[
  x'\in\widetilde P'\setminus\widetilde Q'
  \qquad\text{and}\qquad
  y'\in\widetilde Q'\setminus\widetilde P',
\]
which proves the weak invariant.

Since $(\widetilde P,\widetilde Q)$ is an $xy$-pair,
Equation~\eqref{eq:comp-cnxy} gives
\[
  \bigl|\CN(\widetilde P,\widetilde Q)\bigr|=|N|.
\]
Moreover, $\varphi^\ast$ maps this common-neighbourhood bijectively
onto $\CN(\widetilde P',\widetilde Q')$. Hence
\begin{equation}\label{eq:step-cn}
  \bigl|\CN(\widetilde P',\widetilde Q')\bigr|=|N|.
\end{equation}
Also,
\begin{equation}\label{eq:step-sd}
  \widetilde P'\sd\widetilde Q'
  =
  \{x',y'\}
  \cup
  \bigl(\{t'\}\sd\{r'\}\bigr)
  \cup
  \bigl(\{u'\}\sd\{s'\}\bigr).
\end{equation}
The three sets on the right-hand side are pairwise disjoint.

We first show that it is impossible to have simultaneously
\[
  u'\ne s'
  \qquad\text{and}\qquad
  t'\ne r'.
\]
Indeed, in that case~\eqref{eq:step-sd} gives
\[
  \bigl|\widetilde P'\sd\widetilde Q'\bigr|=6.
\]
By \ref{cn:6} of Lemma~\ref{lem:cn}, this would imply
\[
  \CN(\widetilde P',\widetilde Q')=\emptyset,
\]
contradicting~\eqref{eq:step-cn} and $|N|\ge2$. Therefore,
\begin{equation}\label{eq:step-equality}
  u'=s'
  \qquad\text{or}\qquad
  t'=r'.
\end{equation}

We now prove the remaining assertions.

\begin{itemize}
  \item[\textup{(i)}]
    Suppose first that $xy\in E(G)$. Since
    \[
      \widetilde P\triangle\widetilde Q=\{x,y\},
    \]
    we have $\widetilde P\sim\widetilde Q$, and hence
    \[
      \widetilde P'\sim\widetilde Q'.
    \]
    By the weak invariant, $x'$ and $y'$ belong to
    $\widetilde P'\sd\widetilde Q'$. Since adjacent token
    configurations have symmetric difference of size two, it follows
    that
    \[
      \widetilde P'\sd\widetilde Q'=\{x',y'\}.
    \]
    Therefore,
    \[
      (S'\setminus\{u'\})\cup\{t'\}
      =
      (S'\setminus\{s'\})\cup\{r'\}.
    \]
    Since $u',s'\in S'$ and $t',r'\notin S'$, this equality implies
    \[
      u'=s'
      \qquad\text{and}\qquad
      t'=r'.
    \]

    Suppose now that
    \[
      xy\notin E(G)
      \qquad\text{and}\qquad
      |N|\ge3.
    \]
    By~\ref{m:xy}, we also have $x'y'\notin E(G')$.
    By~\eqref{eq:step-equality}, at least one of the equalities
    \[
      u'=s'
      \qquad\text{or}\qquad
      t'=r'
    \]
    holds. If exactly one of them holds, then
    \[
      \bigl|\widetilde P'\sd\widetilde Q'\bigr|=4.
    \]
    Since $x'$ and $y'$ lie on opposite sides of this symmetric
    difference and $x'y'\notin E(G')$, \ref{cn:4} of
    Lemma~\ref{lem:cn} gives
    \[
      \bigl|\CN(\widetilde P',\widetilde Q')\bigr|\le2.
    \]
    This contradicts~\eqref{eq:step-cn} and $|N|\ge3$. Hence
    \[
      u'=s'
      \qquad\text{and}\qquad
      t'=r'.
    \]

    Thus, if $xy\in E(G)$ or $|N|\ge3$, the pair
    $(\widetilde P,\widetilde Q)$ is strong, with common part
    \[
      \widetilde S'
        =(S'\setminus\{u'\})\cup\{t'\}.
    \]

  \item[\textup{(ii)}]
    Suppose that
    \[
      xy\notin E(G)
      \qquad\text{and}\qquad
      |N|=2.
    \]
    By~\ref{m:xy}, we have $x'y'\notin E(G')$, and
    Equation~\eqref{eq:step-cn} becomes
    \begin{equation}\label{eq:step-cn-two}
      \bigl|\CN(\widetilde P',\widetilde Q')\bigr|=2.
    \end{equation}
    By~\eqref{eq:step-equality}, at least one of
    \[
      u'=s'
      \qquad\text{or}\qquad
      t'=r'
    \]
    holds.

    If both equalities hold, then
    $(\widetilde P,\widetilde Q)$ is strong, with common part
    \[
      \widetilde S'
        =(S'\setminus\{u'\})\cup\{t'\}.
    \]
    This is case \textup{(S)}.

    Suppose next that
    \[
      u'=s'
      \qquad\text{and}\qquad
      t'\ne r'.
    \]
    Then
    \[
      \widetilde P'\setminus\widetilde Q'=\{x',t'\}
      \qquad\text{and}\qquad
      \widetilde Q'\setminus\widetilde P'=\{y',r'\}.
    \]
    Since $x'y'\notin E(G')$, \ref{cn:4} of
    Lemma~\ref{lem:cn}, together with~\eqref{eq:step-cn-two},
    implies that
    \[
      x'r',t'y'\in E(G').
    \]
    By~\ref{m:CN}, it follows that
    \[
      t',r'\in N'.
    \]
    Since $t',r'\notin S'$, we obtain
    \[
      t',r'\in N'\setminus S'.
    \]
    This is case \textup{(A)}.

    Finally, suppose that
    \[
      t'=r'
      \qquad\text{and}\qquad
      u'\ne s'.
    \]
    Then
    \[
      \widetilde P'\setminus\widetilde Q'=\{x',s'\}
      \qquad\text{and}\qquad
      \widetilde Q'\setminus\widetilde P'=\{y',u'\}.
    \]
    Since $x'y'\notin E(G')$, \ref{cn:4} of
    Lemma~\ref{lem:cn}, together with~\eqref{eq:step-cn-two},
    implies that
    \[
      x'u',s'y'\in E(G').
    \]
    By~\ref{m:CN}, it follows that
    \[
      u',s'\in N'.
    \]
    Since $u',s'\in S'$, we obtain
    \[
      u',s'\in N'\cap S'.
    \]
    This is case \textup{(B)}.

    The three cases \textup{(S)}, \textup{(A)}, and \textup{(B)}
    are mutually exclusive and, by~\eqref{eq:step-equality},
    exhaustive.
\end{itemize}
\end{proof}

\medskip\noindent\textbf{Chains and the easy case.}

We call a sequence of $xy$-pairs, each obtained from the previous
one by a move, a \emph{chain}. The assignment
\[
  (P,Q)\longmapsto S=P\cap Q
\]
identifies chains with walks in
$F_{k^\ast-1}(G\setminus\{x,y\})$: a move
\[
  (P,Q)\xrightarrow{\,u\to z\,}(\widetilde P,\widetilde Q)
\]
uses an edge $uz$ of $G\setminus\{x,y\}$, and conversely.

We shall use the following standard consequence of the connectivity
of token graphs of connected graphs: if $H$ is a connected subgraph
of $G\setminus\{x,y\}$, and all tokens outside $H$ are kept fixed,
then the tokens inside $H$ can be moved from any configuration to any
other configuration of the same size by moves inside $H$.

\begin{lemma}\label{lem:comp-easy}
If
\[
  xy\in E(G)
  \qquad\text{or}\qquad
  |N|\ge3,
\]
then $v$ and $w$ lie in different components of
$G\setminus\{x,y\}$.
\end{lemma}

\begin{proof}
Suppose, for a contradiction, that $v$ and $w$ lie in the same
component $H$ of $G\setminus\{x,y\}$.

The $xy$-pairs $(A_5,A_4)$ and $(A_2,A_6)$ have common parts
\[
  X^\ast\cup\{v\}
  \qquad\text{and}\qquad
  X^\ast\cup\{w\},
\]
respectively. Since $v$ and $w$ lie in the same component of
$G\setminus\{x,y\}$, these two common parts place the same number of
tokens in every component of $G\setminus\{x,y\}$. Hence there is a
chain
\[
  (P_1,Q_1),\ldots,(P_{\ell+1},Q_{\ell+1})
\]
from
\[
  (P_1,Q_1)=(A_5,A_4)
\]
to
\[
  (P_{\ell+1},Q_{\ell+1})=(A_2,A_6).
\]

The pair $(A_5,A_4)$ is strong, since
\[
  \varphi^\ast(A_5)=A_5'=X'\cup\{x',v'\}
  \qquad\text{and}\qquad
  \varphi^\ast(A_4)=A_4'=X'\cup\{v',y'\}.
\]
By \textup{(i)} of Lemma~\ref{lem:comp-step}, every pair in the
chain is strong. In particular, the final pair $(A_2,A_6)$ is
strong. Thus
\[
  x'\in\varphi^\ast(A_2)=A_2'=X'\cup\{y',w'\};
\]
a contradiction.
\end{proof}

\medskip\noindent\textbf{The hard case.}

By Lemma~\ref{lem:comp-easy}, it what follows assume that
\begin{equation}\label{eq:comp-rigid}
  xy\notin E(G)
  \qquad\text{and}\qquad
  |N|=2.
\end{equation}
By~\ref{m:xy}, we also have
\[
  x'y'\notin E(G').
\]
Moreover, the edges $xv$ and $yw$ are the original disjoint edges,
and $x'v'$ and $y'w'$ are edges of the cycle $C'$. Hence, by
\ref{m:CN},
\begin{equation}\label{eq:comp-rigidN}
  N=\{v,w\}
  \qquad\text{and}\qquad
  N'=\{v',w'\}.
\end{equation}
We shall also use that
\[
  vw\notin E(G)
  \qquad\text{and}\qquad
  v'w'\notin E(G').
\]

\medskip\noindent\textbf{Aligned pairs.}

Define
\[
  n(P,Q):=|S\cap\{v,w\}|\in\{0,1,2\}.
\]
If $(P,Q)$ is strong, define
\[
  m(P,Q):=|S'\cap\{v',w'\}|\in\{0,1,2\}.
\]

Since $N=\{v,w\}$, Equation~\eqref{eq:comp-cnxy} says that
$\CN(P,Q)$ has exactly two elements. More explicitly:
\[
\begin{array}{lll}
n=0:
  & Z_v:=S\cup\{v\},
  & Z_w:=S\cup\{w\};\\[2pt]
n=1:
  & Z:=S\cup\{\bar c\},
  & D:=(S\setminus\{c\})\cup\{x,y\};\\[2pt]
n=2:
  & D_v:=(S\setminus\{v\})\cup\{x,y\},
  & D_w:=(S\setminus\{w\})\cup\{x,y\}.
\end{array}
\]
In the middle row, $c$ denotes the unique element of
$S\cap\{v,w\}$, and $\bar c$ denotes the other element of
$\{v,w\}$.

Note that the configurations of the form $S\cup\{c\}$ are
$Zero$-configurations of the pair, and the configurations of the
form
$
  (S\setminus\{c\})\cup\{x,y\}$
are $Double$-configurations of the pair.

For a strong pair, the same description applies to
$\CN(P',Q')$ on the primed side, with
\[
  S',\quad m(P,Q),\quad \{v',w'\},\quad \{x',y'\}
\]
in place of
\[
  S,\quad n(P,Q),\quad \{v,w\},\quad \{x,y\},
\]
respectively.
We call the corresponding configurations the
\emph{$Zero'$-configurations} and \emph{$Double'$-configurations}.

\begin{definition}\label{def:comp-aligned}
Let $(P,Q)$ be an $xy$-pair. We say that $(P,Q)$ is
\emph{aligned} if the following hold:
\begin{enumerate}[label=\textup{(\alph*)},leftmargin=2.6em]
  \item $(P,Q)$ is strong;

  \item $m(P,Q)=n(P,Q)$;

  \item $\varphi^\ast$ maps each $Zero$-configuration of
  $\CN(P,Q)$ to a $Zero'$-configuration of $\CN(P',Q')$, and each
  $Double$-configuration to a $Double'$-configuration.
\end{enumerate}
For $n(P,Q)=0$ or $2$, condition \textup{(c)} follows from
condition \textup{(b)}, since when $m(P,Q)=0$ both primed common
neighbours are $Zero'$-configurations, and when $m(P,Q)=2$ both are
$Double'$-configurations. Thus condition \textup{(c)} has content
only when
\[
  n(P,Q)=m(P,Q)=1.
\]
\end{definition}

\begin{remark}\label{rem:comp-start}
The pair $(A_5,A_4)$ is aligned.
\end{remark}

\medskip\noindent\textbf{Move types.}

Let
\[
  (P,Q)\xrightarrow{\,u\to z\,}(\widetilde P,\widetilde Q).
\]
  We classify the move
$u\to z$ as follows:
\begin{itemize}[leftmargin=2.2em]
  \item it is \emph{safe} if
  \[
    u,z\notin\{v,w\};
  \]
  in this case
  \[
    n(\widetilde P,\widetilde Q)=n(P,Q);
  \]

  \item it is a \emph{departure} if
  \[
    u\in\{v,w\};
  \]
  in this case $z\notin\{v,w\}$, since $uz\in E(G)$ and
  $vw\notin E(G)$, and therefore
  \[
    n(\widetilde P,\widetilde Q)=n(P,Q)-1;
  \]

  \item it is an \emph{arrival} if
  \[
    z\in\{v,w\};
  \]
  in this case $u\notin\{v,w\}$, again because $uz\in E(G)$ and
  $vw\notin E(G)$, and therefore
  \[
    n(\widetilde P,\widetilde Q)=n(P,Q)+1.
  \]
\end{itemize}
This is a complete and disjoint
classification of the moves in the hard case.

\medskip\noindent\textbf{The Immunity Lemma.}
\begin{lemma}[Immunity Lemma]\label{lem:comp-imm}
Let $(P,Q)$ be an aligned $xy$-pair and suppose that
\[
  (P,Q)\xrightarrow{\,u\to z\,}(\widetilde P,\widetilde Q).
\]
Then the following hold.
\begin{enumerate}[label=\textup{(\roman*)},leftmargin=2.6em]
  \item\label{imm:safe}
  If the move is safe, then $(\widetilde P,\widetilde Q)$ is aligned.

  \item\label{imm:dep}
  If the move is a departure and $n(P,Q)=1$, then
  $(\widetilde P,\widetilde Q)$ is aligned.

  \item\label{imm:arr}
  If the move is an arrival and $n(P,Q)=1$, then
  $(\widetilde P,\widetilde Q)$ is aligned.
\end{enumerate}
We call the move steps covered by \textup{(i)}--\textup{(iii)}
\emph{immune}. The remaining move steps, namely arrivals at
$n(P,Q)=0$ and departures at $n(P,Q)=2$, are called
\emph{vulnerable}.
\end{lemma}

\begin{proof}
By the definition of move types, safe moves satisfy
$\widetilde n=n(P,Q)$, departures satisfy $n(\widetilde P,\widetilde Q)=n(P,Q)-1$, and arrivals
satisfy $n(\widetilde P,\widetilde Q)=n(P,Q)+1$.

Let
\[
  \widetilde P':=\varphi^\ast(\widetilde P)
  \qquad\text{and}\qquad
  \widetilde Q':=\varphi^\ast(\widetilde Q).
\]
Since $(P,Q)$ is aligned, it is strong and
\[
  m(P,Q)=n(P,Q).
\]
Thus Lemma~\ref{lem:comp-step} applies to the move
\[
  (P,Q)\xrightarrow{\,u\to z\,}(\widetilde P,\widetilde Q).
\]
Let $u',s',t',r'$ be as in Lemma~\ref{lem:comp-step}.

We first record a consequence of the alternatives in
Lemma~\ref{lem:comp-step}. In case \textup{(A)}, we have
\[
  t'\ne r'
  \qquad\text{and}\qquad
  t',r'\in N'\setminus S'.
\]
Since $N'=\{v',w'\}$, this forces
\[
  S'\cap N'=\emptyset,
\]
and hence $m(P,Q)=0$. Similarly, in case \textup{(B)}, we have
\[
  u'\ne s'
  \qquad\text{and}\qquad
  u',s'\in N'\cap S'.
\]
Since $N'=\{v',w'\}$, this forces
\[
  N'\subseteq S',
\]
and hence $m(P,Q)=2$.

Suppose that \[n(P,Q)=1.\]
Since $(P,Q)$ is strong, we have that $m(P,Q)=1$. Thus,  neither case
\textup{(A)} nor case \textup{(B)} of
Lemma~\ref{lem:comp-step} can occur.
Therefore,

\begin{quote}
\emph{ if $n(P,Q)=1$, then  case \textup{(S)}
occurs and $(\widetilde P,\widetilde Q)$ is strong.}
\end{quote}
We now record the old-new adjacency patterns that will be used
below. Recall that a Zero-configuration contains neither $x$ nor
$y$, while a Double-configuration contains both $x$ and $y$.
Consequently,
\begin{quote}\emph{
a Zero-configuration is never adjacent to a
Double-configuration.}
\end{quote}
The same observation applies on the primed
side.

Let
\[
  \widetilde S:=(S\setminus\{u\})\cup\{z\}.
\]
Suppose that the move is safe. Then neither $u$ nor $z$
belongs to $\{v,w\}$, so $S$ and $\widetilde S$ contain the same
vertices of $\{v,w\}$. Therefore the common neighbours of
$\widetilde P$ and $\widetilde Q$ are obtained from the common
neighbours of $P$ and $Q$ by applying the same move $u\to z$, with
the same type and the same label. In particular, every old common
neighbour has a new common neighbour adjacent to it, and every new
common neighbour has an old common neighbour adjacent to it.

Thus,
\begin{quote}
\emph{if the move is safe and $n(P,Q)=1$, then  the old Zero-configuration is adjacent to the new
Zero-configuration, and  the old Double-configuration is adjacent to
the new Double-configuration.}
\end{quote}

 Let
\[
  c\in S\cap\{v,w\}
  \qquad\text{and}\qquad
  \bar c\in\{v,w\}\setminus\{c\}.
\]
The common neighbours of $P$ and $Q$ are the
Zero-configuration
\[
  Z:=S\cup\{\bar c\}
\]
and the Double-configuration
\[
  D:=(S\setminus\{c\})\cup\{x,y\}.
\]

Suppose that the move is a departure, then $u=c$ and
$n(\widetilde P,\widetilde Q)=0$. Hence both common neighbours of
$\widetilde P$ and $\widetilde Q$ are Zero-configurations.
The old Zero-configuration $Z$ is adjacent to the Zero-configuration
\[
  (S\setminus\{c\})\cup\{z,\bar c\},
\]
since their symmetric difference is $\{c,z\}$ and $cz=uz\in E(G)$.
On the other hand, the old Double-configuration $D$ has no new
common neighbour adjacent to it, since every new common neighbour is
a Zero-configuration. Thus,
\begin{quote}
\emph{if the move is a departure and $n(P,Q)=1$, then the old
Zero-configuration is adjacent to a new Zero configuration, while
the old Double-configuration is not adjacent to a new common neighbour. }
\end{quote}

Suppose that the move is an arrival. We have that $z=\bar c$ and
$n(\widetilde P,\widetilde Q)=2$. Hence both common neighbours of
$\widetilde P$ and $\widetilde Q$ are Double-configurations.
Thus,
\begin{quote}
\emph{if the move is an arrival and  $n(P,Q)=1$, then  the old
Double-configuration is adjacent to a new Double-configuration,
while the old Zero-configuration is not adjacent to a new common neighbour.}
\end{quote}

We now prove \textup{(i)}, \textup{(ii)} and \textup{(iii)}.
We shall use the following observation throughout the proof:

\begin{quote} \emph{
since
$\varphi^\ast$ is an isomorphism and maps
$\CN(P,Q)$ onto $\CN(P',Q')$, and
$\CN(\widetilde P,\widetilde Q)$ onto
$\CN(\widetilde P',\widetilde Q')$, adjacency between old and new
common neighbours is preserved by $\varphi^\ast$.}
\end{quote}

\begin{itemize}
  \item[\textup{(i)}]
  Suppose that the move is safe.

  \begin{itemize}

  \item Suppose that $n(P,Q)=0$.

  We have that $m(P,Q)=0$ and
  \[
    S'\cap N'=\emptyset;
  \]
  thus, case \textup{(B)} of Lemma~\ref{lem:comp-step} cannot occur,
  because it would require two distinct vertices in $N'\cap S'$.

  Suppose that case \textup{(A)} occurs. Then
  \[
    t'\ne r'
    \qquad\text{and}\qquad
    t',r'\in N'\setminus S'.
  \]
  Since $N'=\{v',w'\}$, we have
  \[
    \{t',r'\}=\{v',w'\}.
  \]

  The old primed common neighbours are the Zero-configurations
  \[
    S'\cup\{v'\}
    \qquad\text{and}\qquad
    S'\cup\{w'\}.
  \]
  Note that
  \[
    W':=(S'\setminus\{u'\})\cup\{x',y'\}
  \]
  is a common neighbour of $\widetilde P'$ and $\widetilde Q'$.
  However, for every $e'\in\{v',w'\}$,
  \[
    W'\sd\bigl(S'\cup\{e'\}\bigr)
    =
    \{u',x',y',e'\}.
  \]
  Thus, $W'$ has no neighbour in $\CN(P',Q')$, contradicting the
  safe-move pattern. This implies that case \textup{(S)} occurs and $(\widetilde P,\widetilde Q)$ is
  strong; in particular $(a)$ of the definition of aligned pair holds for $(\widetilde P,\widetilde Q)$.

  The primed common part of $(\widetilde P,\widetilde Q)$ is equal to
  \[
    \widetilde S'=(S'\setminus\{u'\})\cup\{t'\}.
  \]
  We claim that
  \[
    t'\notin\{v',w'\}.
  \]
  Suppose otherwise. Then the new primed Double-configuration
  \[
    (\widetilde S'\setminus\{t'\})\cup\{x',y'\}
    =
    (S'\setminus\{u'\})\cup\{x',y'\}
  \]
  has no neighbour in $\CN(P',Q')$, by the same computation as
  above. This again contradicts the safe-move pattern. Hence
  $t'\notin\{v',w'\}$. Therefore,
  \[
    \widetilde S'\cap N'=\emptyset,
  \]
  and
  \[
    m(\widetilde P,\widetilde Q)=0=n(\widetilde P,\widetilde Q).
  \]
  thus, $(b)$ and $(c)$ in the definition of aligned pair hold for
  $(\widetilde P,\widetilde Q)$,

  \item   Suppose that $n(P,Q)=1$.

  Let
  \[
    d'\in S'\cap\{v',w'\},
    \qquad
    \bar d'\in\{v',w'\}\setminus\{d'\}.
  \]
  Since $m(P,Q)=1$, neither case \textup{(A)} nor case
  \textup{(B)} of Lemma~\ref{lem:comp-step} can occur, and so case
  \textup{(S)} occurs. Thus, $(\widetilde P,\widetilde Q)$ is
  strong and $(a)$ in the definition of aligned pair holds for  $(\widetilde P,\widetilde Q)$.

  The primed common part of $(\widetilde P,\widetilde Q)$ is equal to
  \[
    \widetilde S'=(S'\setminus\{u'\})\cup\{t'\}.
  \]

  We claim that
  \[
    u'\ne d'
    \qquad\text{and}\qquad
    t'\notin\{v',w'\}.
  \]

 Since $v'$ is not adjacent to $w'$, it cannot be the case that
 \[
    u'=d'
    \qquad\text{and}\qquad
    t'=\bar d'.
  \]

  Suppose that
  \[
    u'\ne d'
    \qquad\text{and}\qquad
    t'\in\{v',w'\}.
  \]
   Since $d'\in S'$ and $t'\notin S'$, the only way that
  $t'\in\{v',w'\}$ is if $t'=\bar d'$.

  Note that $m(\widetilde P,\widetilde Q)=2$. The old primed
  Zero-configuration is
  \[
    Z'=S'\cup\{\bar d'\}.
  \]
  The two new primed common neighbours are
  \[
    (S'\setminus\{u',d'\})\cup\{\bar d',x',y'\}
  \]
  and
  \[
    (S'\setminus\{u'\})\cup\{x',y'\}.
  \]
  Each has symmetric difference of size four with $Z'$. Hence $Z'$
  has no new primed neighbour; which contradicts the safe-move pattern.

  Suppose that
  \[
    u'=d'
    \qquad\text{and}\qquad
    t'\notin\{v',w'\}.
  \]
  Then $m(\widetilde P,\widetilde Q)=0$. The old primed
  Double-configuration is
  \[
    D'=(S'\setminus\{d'\})\cup\{x',y'\}.
  \]
  The two new primed common neighbours are
  \[
    S'\cup\{t'\}
  \]
  and
  \[
    (S'\setminus\{d'\})\cup\{t',\bar d'\}.
  \]
  Each has symmetric difference of size four with $D'$. Hence $D'$
  has no new primed neighbour; which contradicts the safe-move
  pattern.

  Therefore,
  \[
    u'\ne d'
    \qquad\text{and}\qquad
    t'\notin\{v',w'\}.
  \]
  Hence
  \[
    \widetilde S'\cap N'=\{d'\},
  \]
  and so
  \[
    m(\widetilde P,\widetilde Q)=1=n(\widetilde P,\widetilde Q),
  \]
and $(b)$ in the definition of aligned pair holds for $(\widetilde P,\widetilde Q)$.

 Since $(P,Q)$ is aligned, the old
  Zero-configuration maps to the old primed Zero-configuration
  \[
    Z'=S'\cup\{\bar d'\}.
  \]
  The new primed Zero-configuration is
  \[
    \widetilde Z'
      =
      \widetilde S'\cup\{\bar d'\}
      =
      (S'\setminus\{u'\})\cup\{t',\bar d'\}.
  \]
  Moreover,
  \[
    Z'\sd\widetilde Z'=\{u',t'\},
  \]
  and this is an edge of $G'$, since $\widetilde P'\sim P'$.
  Therefore $Z'$ is adjacent to $\widetilde Z'$. The new primed
  Double-configuration is not adjacent to $Z'$, since a
  Zero-configuration is not adjacent to a Double-configuration.

  On the unprimed side, the old Zero-configuration is adjacent to
  the new Zero-configuration and not adjacent to the new
  Double-configuration. Since $\varphi^\ast$ preserves adjacency and
  maps common neighbours to common neighbours, the new
  Zero-configuration must map to $\widetilde Z'$. Consequently, the
  new Double-configuration maps to the new primed
  Double-configuration. Therefore $(c)$ holds and $(\widetilde P,\widetilde Q)$ is
  aligned.

  \item Suppose that $n(P,Q)=2$.

 Since $(P,Q)$ is aligned, we have
  $m(P,Q)=2$, and hence
  \[
    N'\subseteq S'.
  \]
  Case \textup{(A)} of Lemma~\ref{lem:comp-step} cannot occur,
  because it would require two distinct vertices in $N'\setminus S'$.
  Suppose that case \textup{(B)} occurs. Then
  \[
    u'\ne s'
    \qquad\text{and}\qquad
    u',s'\in N'\cap S'.
  \]
  Since $N'=\{v',w'\}$, we have
  \[
    \{u',s'\}=\{v',w'\}.
  \]

  Thus the old primed common neighbours are
  \[
    (S'\setminus\{v'\})\cup\{x',y'\}
    \qquad\text{and}\qquad
    (S'\setminus\{w'\})\cup\{x',y'\}.
  \]

  In this case
  \[
    W':=S'\cup\{t'\}
  \]
  is a common neighbour of $\widetilde P'$ and $\widetilde Q'$.
  However, for every $e'\in\{v',w'\}$,
  \[
    W'\sd\bigl((S'\setminus\{e'\})\cup\{x',y'\}\bigr)
    =
    \{e',t',x',y'\}.
  \]
  Thus $W'$ has no neighbour in $\CN(P',Q')$, contradicting the
  safe-move pattern. This implies that case \textup{(S)} occurs.
   Thus, $(\widetilde P,\widetilde Q)$ is
  strong and $(a)$ in the definition of aligned pair holds for  $(\widetilde P,\widetilde Q)$.

  The primed common part of $(\widetilde P,\widetilde Q)$ is equal to
  \[
    \widetilde S'=(S'\setminus\{u'\})\cup\{t'\}.
  \]
  We claim that
  \[
    u'\notin\{v',w'\}.
  \]
  Suppose otherwise. Then the new primed Zero-configuration
  \[
    \widetilde S'\cup\{u'\}=S'\cup\{t'\}
  \]
  has no neighbour in $\CN(P',Q')$, by the same computation as
  above. This contradicts the safe-move pattern. Hence
  $u'\notin\{v',w'\}$. Since $N'\subseteq S'$ and
  $t'\notin S'$, we also have $t'\notin\{v',w'\}$. Therefore
  \[
    N'\subseteq \widetilde S',
  \]
  and so
  \[
    m(\widetilde P,\widetilde Q)=2=n(\widetilde P,\widetilde Q);
  \]
  thus, $(b)$ and $(c)$ in the definition of aligned pair hold for $(\widetilde P,\widetilde Q)$,
  and $(\widetilde P,\widetilde Q)$ is aligned.

\end{itemize}

  \item[\textup{(ii)}] Suppose that the move is a departure and $n(P,Q)=1$.

  We have that $n(\widetilde P,\widetilde Q)=0$.
 Since $(P,Q)$ is aligned and $n(P,Q)=1$, we have
\(
  m(P,Q)=1.
\)
If case \textup{(A)} of Lemma~\ref{lem:comp-step} occurred, then
\[
  t',r'\in N'\setminus S'
  \qquad\text{and}\qquad
  t'\ne r',
\]
forcing $S'\cap N'=\emptyset$, and hence $m(P,Q)=0$. If case
\textup{(B)} occurred, then
\[
  u',s'\in N'\cap S'
  \qquad\text{and}\qquad
  u'\ne s',
\]
forcing $N'\subseteq S'$, and hence $m(P,Q)=2$. Both are impossible.
Therefore, case \textup{(S)} of Lemma~\ref{lem:comp-step} occurs. In particular
$(\widetilde P,\widetilde Q)$ is strong and $(a)$ in the definition of aligned pair holds
for $(\widetilde P,\widetilde Q)$.

The primed common part of $(\widetilde P,\widetilde Q)$ is equal to
\[
    \widetilde S'=(S'\setminus\{u'\})\cup\{t'\}.
  \]

 Let
  \[
    c\in S\cap\{v,w\},
    \qquad
    \bar c\in\{v,w\}\setminus\{c\},
  \]
  and let
  \[
    d'\in S'\cap\{v',w'\},
    \qquad
    \bar d'\in\{v',w'\}\setminus\{d'\}.
  \]


  We claim that
  \[
    u'=d'
    \qquad\text{and}\qquad
    t'\notin\{v',w'\}.
  \]
Note that since $v'$ is not adjacent to $w'$ we cannot have that
  \[
    u'=d'
    \qquad\text{and}\qquad
    t' \in\{v',w'\}.
  \]

  Suppose  that
  \[
    u'\ne d'
    \qquad\text{and}\qquad
    t'\notin\{v',w'\}.
  \]
  Then $m(\widetilde P,\widetilde Q)=1$. The new primed
  Double-configuration is
  \[
    (\widetilde S'\setminus\{d'\})\cup\{x',y'\}.
  \]
  Its symmetric difference with the old primed Double-configuration
  \[
    D'=(S'\setminus\{d'\})\cup\{x',y'\}
  \]
  is $\{u',t'\}$, an edge of $G'$. Hence $D'$ has a new primed
  neighbour, contradicting the departure pattern.

  Suppose  that
  \[
    u'\ne d'
    \qquad\text{and}\qquad
    t'=\bar d'.
  \]
  Then $m(\widetilde P,\widetilde Q)=2$. The old primed
  Zero-configuration
  \[
    Z'=S'\cup\{\bar d'\}
  \]
  has no new primed neighbour, by similar computations; contradicting the departure pattern.

  Therefore,
  \[
    u'=d'
    \qquad\text{and}\qquad
    t'\notin\{v',w'\}.
  \]
  Hence
  \[
    \widetilde S'\cap N'=\emptyset,
  \]
  and so
  \[
    m(\widetilde P,\widetilde Q)=0=n(\widetilde P,\widetilde Q);
  \]
  thus, condition $(b)$ holds for $(\widetilde P,\widetilde Q)$.
  Since $(\widetilde P,\widetilde Q)$ is strong, and both common
  neighbours on each side are Zero-configurations, $(c)$ holds as well and
  $(\widetilde P,\widetilde Q)$ is aligned.

   \item[\textup{(iii)}] Suppose that the move is an arrival and $n(P,Q)=1$.

  We have that $n(\widetilde P,\widetilde Q)=2$.
  Since $(P,Q)$ is aligned and $n(P,Q)=1$, we have
  \(
    m(P,Q)=1.
  \)
  If case \textup{(A)} of Lemma~\ref{lem:comp-step} occurred, then
  \[
    t',r'\in N'\setminus S'
    \qquad\text{and}\qquad
    t'\ne r',
  \]
  forcing $S'\cap N'=\emptyset$, and hence $m(P,Q)=0$. If case
  \textup{(B)} occurred, then
  \[
    u',s'\in N'\cap S'
    \qquad\text{and}\qquad
    u'\ne s',
  \]
  forcing $N'\subseteq S'$, and hence $m(P,Q)=2$. Both are impossible.
  Therefore, case \textup{(S)} of Lemma~\ref{lem:comp-step} occurs. In particular
  $(\widetilde P,\widetilde Q)$ is strong and $(a)$ in the definition of aligned pair holds
  for $(\widetilde P,\widetilde Q)$.

  The primed common part of $(\widetilde P,\widetilde Q)$ is equal to
  \[
    \widetilde S'=(S'\setminus\{u'\})\cup\{t'\}.
  \]

  Let
  \[
    c\in S\cap\{v,w\},
    \qquad
    \bar c\in\{v,w\}\setminus\{c\},
  \]
  and let
  \[
    d'\in S'\cap\{v',w'\},
    \qquad
    \bar d'\in\{v',w'\}\setminus\{d'\}.
  \]

  We claim that
  \[
    u'\ne d'
    \qquad\text{and}\qquad
    t'=\bar d'.
  \]
  Note that since $v'$ is not adjacent to $w'$ we cannot have that
  \[
    u'=d'
    \qquad\text{and}\qquad
    t'\in\{v',w'\}.
  \]

  Suppose that
  \[
    u'\ne d'
    \qquad\text{and}\qquad
    t'\notin\{v',w'\}.
  \]
  Then $m(\widetilde P,\widetilde Q)=1$. The new primed
  Zero-configuration is
  \[
    \widetilde S'\cup\{\bar d'\}.
  \]
  Its symmetric difference with the old primed Zero-configuration
  \[
    Z'=S'\cup\{\bar d'\}
  \]
  is $\{u',t'\}$, an edge of $G'$. Hence $Z'$ has a new primed
  neighbour, contradicting the arrival pattern.

  Suppose that
  \[
    u'=d'
    \qquad\text{and}\qquad
    t'\notin\{v',w'\}.
  \]
  Then $m(\widetilde P,\widetilde Q)=0$. The old primed
  Double-configuration is
  \[
    D'=(S'\setminus\{d'\})\cup\{x',y'\}.
  \]
  The two new primed common neighbours are
  \[
    S'\cup\{t'\}
  \]
  and
  \[
    (S'\setminus\{d'\})\cup\{t',\bar d'\}.
  \]
  Each has symmetric difference of size four with $D'$. Hence $D'$
  has no new primed neighbour, contradicting the arrival pattern.

  Therefore,
  \[
    u'\ne d'
    \qquad\text{and}\qquad
    t'=\bar d'.
  \]
  Hence
  \[
    N'\subseteq\widetilde S',
  \]
  and so
  \[
    m(\widetilde P,\widetilde Q)=2=n(\widetilde P,\widetilde Q);
  \]
  thus, condition $(b)$ holds for $(\widetilde P,\widetilde Q)$.
  Since $(\widetilde P,\widetilde Q)$ is strong, and both common
  neighbours on each side are Double-configurations, $(c)$ holds as well and
  $(\widetilde P,\widetilde Q)$ is aligned.
\end{itemize}
\end{proof}

Since $x'$ and $y'$ are twins in $G'$ with $x'y'\notin E(G')$, we shall also
use a mirror image of the machinery just developed.

\begin{remark}\label{rem:comp-mirror}
By~\ref{m:CN} and~\ref{m:xy}, $x'$ and $y'$ are twins in $G'$ and
$x'y'\notin E(G')$; hence every hypothesis about the primed side used in
Lemmas~\ref{lem:comp-step} and~\ref{lem:comp-imm} is symmetric in $x'$ and
$y'$. Interchanging the names of $x'$ and $y'$ on the primed side therefore
leaves both lemmas valid and replaces the notion of a strong pair by the
following one: an $xy$-pair $(P,Q)$ is \emph{strong$^\ast$} if
$\varphi^\ast(P)\sd\varphi^\ast(Q)=\{x',y'\}$ with $y'\in\varphi^\ast(P)$ and
$x'\in\varphi^\ast(Q)$. Defining \emph{aligned$^\ast$} accordingly, the Step
Lemma then yields the mirrored weak invariant
$y'\in\widetilde P'\setminus\widetilde Q'$, and the Immunity Lemma preserves
aligned$^\ast$ under immune moves.

The pair $(A_2,A_6)$ is aligned$^\ast$: it is the $xy$-pair with
$S=X^\ast\cup\{w\}$, and its images $\varphi^\ast(A_2)=A_2'=X'\cup\{y',w'\}$
and $\varphi^\ast(A_6)=A_6'=X'\cup\{x',w'\}$ exhibit it as strong$^\ast$ with
$S'=X'\cup\{w'\}$ and $n=m=1$; moreover its common neighbours
$A_3=S\cup\{v\}$ and $A_1=(S\setminus\{w\})\cup\{x,y\}$ map to the
$Zero'$-configuration $A_3'=S'\cup\{v'\}$ and the $Double'$-configuration
$A_1'=(S'\setminus\{w'\})\cup\{x',y'\}$.
\end{remark}

We are finally ready to prove~\ref{m:comp} of Theorem~\ref{thm:G}.
Suppose for a contradiction that $v$ and $w$ lie in the same component $H$ of $G\setminus\{x,y\}$;
Let
\[\Gamma:=(v=u_1,u_2,\dots,u_l=w)\] be a shortest path from $v$ to $w$ in $G\setminus\{x,y\}$.

We claim that
\begin{quote}
 \emph{
 for one of the two ordered pairs
\[
((A_5,A_4),(A_2,A_6))
\quad\text{and}\quad
((A_2,A_6),(A_5,A_4)),
\]
there is a chain from the first pair to the second in which every move
except the last is immune, and the last move is vulnerable.
 }
\end{quote}
Let $a_1,\dots,a_r$ be the elements of \[(X^\ast \cup \{v\})\cap \{u_1,\dots,u_l\}\] in the order
as they appear in $\Gamma$; thus $a_1=v$.

Suppose that $a_r\neq u_{l-1}$. Consider the following moves. Move the token at $a_r$ along $\Gamma$
to $u_{l-1}$. For $j=r-1$ to $j=1$ move the token at $a_j$ along $\Gamma$ to $a_{j+1}$. All of these
moves are immune. Finally move the token at $u_{l-1}$ to $u_l=w$, this move is vulnerable.
If $u_2 \neq a_2$, with symmetric arguments we can find a chain from $(A_2,A_6)$ to $(A_5,A_4)$
in which all moves but the last one are immune; and the last move is vulnerable.

Suppose that $a_r=u_{l-1}$ and $u_2=a_2$. Move the token from $a_r=u_{l-1}$ to $w$. This is immune arrival
move since $n(A_5,A_4)=1$. For $i=r-1$ to $i=2$ move the token at $a_i$ along $\Gamma$ to $a_{i+1}$; all of these moves are immune. Finally, move the token at $v=a_1$ to $u_2$, this move is vulnerable.
Let $(P_1,Q_1),\dots,(P_{\ell+1},Q_{\ell+1})$  be such a chain.

Suppose that $(P_1,Q_1)=(A_5,A_4)$. As every move but  the last one is immune, each
member of the path up to (and including) $(P_{\ell},Q_{\ell})$ is aligned; in particular
$(P_{\ell},Q_{\ell})$; this contradicts the weak invariant of Lemma~\ref{lem:comp-step}, since
\[(P_{\ell+1},Q_{\ell+1})=(A_2,A_6), \text{ but } x' \notin A_2' \text{ and } y' \notin A_6'.\]
Suppose instead that $(P_1,Q_1)=(A_2,A_6)$, so that
$(P_{\ell+1},Q_{\ell+1})=(A_5,A_4)$. By Remark~\ref{rem:comp-mirror}, the pair
$(A_2,A_6)$ is aligned$^\ast$. As every move but the last is immune, the
mirrored Immunity Lemma shows that each member of the chain up to and
including $(P_\ell,Q_\ell)$ is aligned$^\ast$, and in particular strong$^\ast$.
Applying the mirrored weak invariant to the last move gives
$y'\in\varphi^\ast(P_{\ell+1})=\varphi^\ast(A_5)=A_5'=X'\cup\{x',v'\}$; but
$y'\notin X'$ and $y'\ne x',v'$, a contradiction.

We have shown that
\begin{quote}
\emph{
 $v,w$ lie in different components of $G\setminus \{x,y\}$.
 }
\end{quote}

By~\ref{sym:inv} of Lemma~\ref{lem:symmetries}, the inverse
$(\varphi^\ast)^{-1}$ realizes the standing configuration with the roles of
the primed and unprimed data interchanged. The argument above, proving that
$v$ and $w$ lie in different components of $G\setminus\{x,y\}$, uses only the
standing configuration together with~\ref{m:xy}--\ref{m:CN}, all of which are
preserved; applied to the hatted data it yields that
\begin{quote}
\emph{
 $v',w'$ lie in different components of $G'\setminus\{x',y'\}$.
 }
\end{quote}

Together, these two statements prove~\ref{m:comp}. With this new structure
we are able to strengthen Lemma~\ref{lem:comp-step}.

\begin{lemma}[Updated Step Lemma]\label{lem:comp-step-updated}
Suppose that $(P,Q)$ is a strong $xy$-pair and that
\[
  (P,Q)\xrightarrow{\,u\to z\,}(\widetilde P,\widetilde Q).
\]
Then there exists $u' \in P'\cap Q'$ and $t' \notin P' \cup Q'$, such that
\[
  \widetilde P'
    =(P'\setminus\{u'\})\cup\{t'\}
\textrm{ and }
  \widetilde Q'
    =(Q'\setminus\{u'\})\cup\{t'\}.
\]
In particular $(\widetilde P,\widetilde Q)$ is strong.
Moreover, the same hold with \emph{strong} replaced by \emph{strong$^\ast$}.
\end{lemma}
\begin{proof}
If $xy\in E(G)$ or $|N|\ge 3$, part~\textup{(i)} of the Step
Lemma~\ref{lem:comp-step} gives $u'=s'$ and $t'=r'$, which is the claim. Assume
then $xy\notin E(G)$ and $|N|=2$, so $N=\{v,w\}$ and $N'=\{v',w'\}$.
By~\ref{m:CN} and~\ref{m:xy}, $N_{G'}(x')=N_{G'}(y')=\{v',w'\}$; and
by~\ref{m:comp}, $v'$ and $w'$ lie in different components of
$G'\setminus\{x',y'\}$, so any common neighbour of $v'$ and $w'$ other than
$x',y'$ would lie in both of those components---impossible. Hence
\[
  N_{G'}(v')\cap N_{G'}(w')=\{x',y'\}.
\]
By the Step Lemma~\ref{lem:comp-step} exactly one of \textup{(S)}, \textup{(A)},
\textup{(B)} occurs, where $u',s'\in P'\cap Q'$, $t',r'\notin P'\cup Q'$, and the
moves $\widetilde P'\sim P'$, $\widetilde Q'\sim Q'$ read $u'\to t'$, $s'\to r'$
(so $u't',s'r'\in E(G')$). In case~\textup{(A)}, $u'=s'$ and
$\{t',r'\}=\{v',w'\}$, so $u'$ is adjacent to both $v'$ and $w'$ and thus
$u'\in\{x',y'\}$, contradicting $u'\in P'\cap Q'$. In case~\textup{(B)}, $t'=r'$
and $\{u',s'\}=\{v',w'\}$, so $t'$ is adjacent to both $v'$ and $w'$ and thus
$t'\in\{x',y'\}$, contradicting $t'\notin P'\cup Q'$. Hence case~\textup{(S)}
occurs: $u'=s'$ and $t'=r'$, giving
$\widetilde P'=(P'\setminus\{u'\})\cup\{t'\}$ and
$\widetilde Q'=(Q'\setminus\{u'\})\cup\{t'\}$. Since $x'\in\widetilde P'$ and
$y'\in\widetilde Q'$, the pair $(\widetilde P,\widetilde Q)$ is strong. The
statement for strong$^\ast$ follows by interchanging $x'$ and $y'$
(Remark~\ref{rem:comp-mirror}).
\end{proof}

\section{Untwisting}\label{sec:untwist}
We are ready to prove Theorem~\ref{thm:token}, \ref{thm:untwist_cycle} and~\ref{thm:twistfree}, which we restate for convenience

\tokencorr*

\begin{proof}
We prove
$\varphi^\ast(V_{x,1})=V_{x',1}'$ in detail; the other three equalities need only
small changes, indicated at the end.

Let $H_{x,1}$ be the subgraph of $F_{k^\ast}(G)$ induced by $V_{x,1}$. Any two
members of $V_{x,1}$ contain $x$ and have the same number of tokens in every
component of $G\setminus\{x,y\}$; moving a token within a component keeps a
configuration in $V_{x,1}$, and by the connectivity of the token graphs of the
(connected) components these moves join any two members. Hence $H_{x,1}$ is
connected. (Indeed $H_{x,1}\cong F_{k_1}(G_1)\,\square\cdots\square\,
F_{k_r}(G_r)$, with $k_i=|A_5\cap G_i|$.)

Let $B\in V_{x,1}$ and $\bar B:=(B\setminus\{x\})\cup\{y\}$; then $\bar B\in
V_{y,1}$, as $B$ and $\bar B$ have the same token counts in every component.
Choose a path $A_5=P_0,P_1,\dots,P_m=B$ in $H_{x,1}$ and set
$Q_i:=(P_i\setminus\{x\})\cup\{y\}$, so $Q_0=A_4$ and $Q_m=\bar B$. Each step
$P_i\to P_{i+1}$ moves a token other than $x$, hence a token of the common part
of the $xy$-pair $(P_i,Q_i)$; thus
\[
  (A_5,A_4)=(P_0,Q_0)\longrightarrow(P_1,Q_1)\longrightarrow\cdots
  \longrightarrow(P_m,Q_m)=(B,\bar B)
\]
is a chain of $xy$-pairs.

The pair $(A_5,A_4)$ is strong, since $\varphi^\ast(A_5)=A_5'=X'\cup\{x',v'\}$ and
$\varphi^\ast(A_4)=A_4'=X'\cup\{v',y'\}$. Write $S_i':=\varphi^\ast(P_i)\cap
\varphi^\ast(Q_i)$, so $S_0'=X'\cup\{v'\}$. By the Updated Step
Lemma~\ref{lem:comp-step-updated} every $(P_i,Q_i)$ is strong and
$S_{i+1}'=(S_i'\setminus\{u_i'\})\cup\{t_i'\}$ with $u_i',t_i'\notin\{x',y'\}$ and
$u_i't_i'\in E(G')$. Since $u_i',t_i'\ne x',y'$, the edge $u_i't_i'$ lies in
$G'\setminus\{x',y'\}$, so $u_i'$ and $t_i'$ belong to the \emph{same} component
of $G'\setminus\{x',y'\}$; hence replacing $u_i'$ by $t_i'$ leaves the token
count of every component of $G'\setminus\{x',y'\}$ unchanged. Therefore $S_m'$
has the same per-component profile as $S_0'$.
As $(B,\bar B)$ is strong, $\varphi^\ast(B)=S_m'\cup\{x'\}$; hence
$x'\in\varphi^\ast(B)$ and $|\varphi^\ast(B)\cap G_j'|=|S_m'\cap G_j'|=|S_0'\cap
G_j'|=|A_5'\cap G_j'|$ for every $j$. That is, $\varphi^\ast(B)\in V_{x',1}'$, so
$\varphi^\ast(V_{x,1})\subseteq V_{x',1}'$.

By~\ref{sym:inv} of Lemma~\ref{lem:symmetries} the same argument applies to
$(\varphi^\ast)^{-1}$ with the primed and unprimed data interchanged; its base
pair is $(A_5',A_4')$, and it gives $(\varphi^\ast)^{-1}(V_{x',1}')\subseteq
V_{x,1}$, i.e.\ $V_{x',1}'\subseteq\varphi^\ast(V_{x,1})$. Thus
$\varphi^\ast(V_{x,1})=V_{x',1}'$.

Reading the same chain from the $Q$-side gives $\varphi^\ast(V_{y,1})=V_{y',1}'$.
For the pair $V_{x,2},V_{y,2}$ the base pair is $(A_2,A_6)$, which is \emph{strong$^\ast$} (as
$\varphi^\ast(A_2)=A_2'=X'\cup\{y',w'\}$, so $y'\in\varphi^\ast(A_2)$); running
the strong$^\ast$ form of the Updated Step Lemma along paths in $H_{x,2}$ and
$H_{y,2}$ yields $\varphi^\ast(V_{x,2})=V_{y',2}'$ and
$\varphi^\ast(V_{y,2})=V_{x',2}'$.

Since $\varphi^\ast$ is $\varphi$ or $\varphi\circ\mathfrak{c}$ by the choice of
the standing configuration, these four equalities give case~$(a)$, respectively
case~$(b)$, of the statement.
\end{proof}


\untwistcycle*

\begin{proof}
By~\ref{m:comp}, $\{x,y\}$ is a $2$-cut of $G$ whose two vertices have the same
neighbours in $G\setminus\{x,y\}$. Put $\sigma_i:=|A_2\cap G_i|$, and let $\psi$
be the map on $V(F_{k^\ast}(G))$ that sends each configuration $A$ with
$|A\cap\{x,y\}|=1$ and $|A\cap G_i|=\sigma_i$ for all $i$ to the configuration
obtained by moving its token on $\{x,y\}$ to the other vertex of $\{x,y\}$, and
fixes every other configuration. Since $x$ and $y$ have the same neighbours in
$G\setminus\{x,y\}$, $\psi$ is an automorphism of $F_{k^\ast}(G)$~\cite{cut_aut},
so $\varphi^\ast\circ\psi$ is an isomorphism from $F_{k^\ast}(G)$ to $F_{k'}(G')$.

It is readily verified that $(A_1,A_2,A_3,A_4)$ is the $4$-cycle generated by the
disjoint edges $xv$ and $yw$, and that among these four configurations only $A_2$
has a single token on $\{x,y\}$ and profile $\sigma$ ($A_1$ has two such tokens,
$A_3$ has none, and $A_4$ carries $v$ in $G_1$ rather than $w$ in $G_2$, so its
profile differs since $v$ and $w$ lie in different components). Thus $\psi$ fixes
$A_1,A_3,A_4$ and sends $A_2$ to $A_6$, and therefore $\varphi^\ast\circ\psi$ maps
$(A_1,A_2,A_3,A_4)$ to
\[
  (A_1',A_6',A_3',A_4')=\bigl(X'\cup\{x',y'\},\ X'\cup\{x',w'\},\
  X'\cup\{v',w'\},\ X'\cup\{v',y'\}\bigr).
\]
The consecutive symmetric differences of these four sets are the edges
$y'w',\,x'v',\,y'w',\,x'v'$ of $C'$; hence the image is a $4$-cycle generated by
the disjoint edges $x'v'$ and $y'w'$.

It remains to show that $\varphi^\ast\circ\psi$ twists no disjoint-edge
$4$-cycle that $\varphi^\ast$ does not; since the complement preserves the type
of an induced $4$-cycle, the same then holds with $\varphi$ in place of
$\varphi^\ast$. Let $C$ be a disjoint-edge $4$-cycle twisted by
$\varphi^\ast\circ\psi$. By Theorem~\ref{thm:G} the two vertices of $C$ that
carry a cut token determine a $2$-cut $S$ of $G$ with the same neighbours, and
$C$ belongs to the profile-$t$ block of $S$ for a single profile $t$. By
Theorem~\ref{thm:token}, applied at $S$, an isomorphism twists $C$ precisely when
it carries that block to the crossed primed block; so it suffices to check that
$\varphi^\ast\circ\psi$ and $\varphi^\ast$ send the profile-$t$ block of $S$ to
the same primed block.
\begin{itemize}
\item If $S=\{x,y\}$ and $t=\sigma$, then $C$ lies in the block untwisted above,
  so $\varphi^\ast\circ\psi$ does not twist $C$, contrary to assumption; this
  case does not occur.
\item If $S=\{x,y\}$ and $t\ne\sigma$, then $\psi$ fixes every configuration of
  profile $t$, so $\varphi^\ast\circ\psi$ and $\varphi^\ast$ agree on this block.
\item If $S\ne\{x,y\}$, then $\psi$ alters a configuration of the block only by
  exchanging $x$ and $y$; as $\{x,y\}$ and $S$ are distinct $2$-cuts with the
  same neighbours, $x$ and $y$ lie in a single component of $G\setminus
  S$~\cite{cut_aut}, so the exchange fixes the token on $S$ and the number of
  tokens in each component of $G\setminus S$. Hence $\psi$ maps the profile-$t$
  block of $S$ onto itself, and $\varphi^\ast\circ\psi$ carries it to the
  same primed block as $\varphi^\ast$.
\end{itemize}
In the last two cases $\varphi^\ast\circ\psi$ and $\varphi^\ast$ twist $C$
together, so $\varphi^\ast$ twists $C$. This proves the second assertion.
Finally, $\varphi^\ast$ is $\varphi$ or $\varphi\circ\mathfrak{c}$ by the choice
of the standing configuration, so composing $\varphi$ with $\psi$ (together with
the complement in the latter case) gives the required isomorphism.
\end{proof}

Theorem~\ref{thm:twistfree} now follows from successive applications of Theorem~\ref{thm:untwist_cycle}.
\twistfree*
\qed

%
%

 \bibliographystyle{alpha} 
\bibliography{Automorphism}

 \end{document}